\documentclass{article}
\usepackage[hmargin=2.5cm,vmargin=2.5cm]{geometry}
\usepackage{import}
\usepackage{AHWpreamble}
\usepackage[numbers]{natbib}
\usepackage{authblk}
\providecommand{\keywords}[1]
{
  \small	
  \textbf{\textit{Keywords---}} #1
}
\renewcommand{\Op}[1]{\mathbf{#1}}
\newcommand{\Fall}{\Op{F}}
\newcommand{\Gall}{\Op{G}}
\newcommand{\GHall}{\Op{M}}
\newcommand{\CDC}{\mathbf{\Gamma}}

\newcommand{\Gdall}{\Op{H}}
\newcommand{\Gdrift}[1]{\Op{G}^{(D)}_{#1}}
\newcommand{\Gdiff}[1]{\Op{G}^{(B)}_{#1}}
\newcommand{\Gjump}[1]{\Op{G}^{(J)}_{#1}}

\newcommand{\Gjumpneg}[1]{\Op{G}^{(J-)}_{#1}}

\newcommand{\typeS}{\mathcal{T}} 
\newcommand{\DWG}{\Op{W}}

\newcommand{\PSSMPG}{\Op{K}}

\newcommand{\LevyMeasure}{\mathtt{\Pi}}
\newcommand{\FVG}{\Op{R}}
\newcommand{\SSS}{\Op{S}} 
\newcommand{\perms}{\mathtt{Perms}}
\newcommand{\drift}{\kappa}
\newcommand{\driftF}{d}

\newcommand{\nun}{\nu^{(n)}}

\newcommand{\xin}{\xi^{(n)}}

\newcommand{\V}{\mathbb{V}}

\newcommand{\abstime}{\mathtt{\tau}}
\begin{document}
\title{The Lamperti transformation in the infinite-dimensional setting, self-similar populations, and coalescents.}

\author[1]{Arno Siri-Jégousse}
\author[2]{Alejandro H. Wences}
\affil[1,2]{IIMAS, Universidad Nacional Aut\'onoma de M\'exico}

\maketitle
\begin{abstract}
{We propose a change in focus from the prevalent paradigm based on the branching property as a tool to analyze the structure of population models, to one based on the self-similarity property, which we also introduce for the first time in the setting of measure-valued processes. 
By extending the well-known Lamperti transformation into the infinite dimensional setting, we were able to embed and extend known results in population genetics within the self-similarity framework: we describe the frequency process of a larger class of measure-valued SS populations in terms of general Lambda Fleming-Viot processes.  Our results demonstrate the potential power of the self-similar perspective for the study
of populations whose total size varies stochastically over time, and in which the reproduction dynamics of the individuals are not independent from one another but are modulated by the total size of the population, allowing for more complex and realistic models. We also uncover a new duality relation between measure-valued processes and Lambda-coalescents which extends the well-known duality relation between Lambda Fleming-Viot processes and Lambda coalescents.}
\end{abstract}

\section{Introduction and main results}
    Consider a stable measure-valued branching
    process $\stp{\mu_t}$ with stability parameter $\beta\in (0,2]$. Formally, the latter are Markov processes on the state space $\sM(\typeS)$
    of finite positive measures on the type space $\typeS$, and with generator $\Fall$ having one of the two forms:
    \begin{itemize}
\item  
 $
\Fall F(\mu) = \int_\typeS\mu(da)\int_0^\infty h^{-1-\beta}dh \bigg\{ F\left(\mu + h \delta_a\right)
     - F(\mu) -  h F'(\mu; a)\bigg\},
    $
     for $\beta\in(0,2)$, or
\item $\Fall F(\mu) = 
\int_\typeS \mu(da) F''(\mu;a,a), 
 $ 
 which  heuristically corresponds to the case $\beta=2$.\\
  \end{itemize}
  (In the above, for simplicity, we have removed the possible additional drift and constant scaling terms). Here, the derivative $F'(\cdot;a)$ (resp. $F''(\cdot;a,b)$) refers to the simple (resp. double) directional Gateaux derivatives in the direction of $\delta_a$ (resp. $\delta_a$ and then $\delta_b$), as defined in section \ref{Sec:LFV} below. 
  Let $\norm{\mu_t}\coloneqq\mu_t(\typeS)$ be the total population size at time $t$. In the same spirit as in \cite{BertoinLeGall2000,Perkins1992}, the work of \cite{les7}
  characterized 
  the evolution of the frequency of types $\stp{\mu_t/\norm{\mu_t}}$ by the following:
    the authors show that, after the stochastic time change
    $c_{{\beta-1}}(t)=\inf\{s\geq 0\colon \int_0^s \norm{\mu_u}^{1-\beta} du\geq t\},$ the frequency process becomes Markov and
    is a member of the Beta subfamily of $\Lambda$-Fleming-Viot (FV)
    processes \cite{BertoinLeGall2003}. The genealogy of the underlying population can be first understood via the well-known duality relation between $\Lambda$-FV processes  and $\Lambda$-coalescents \cite{BertoinLeGall2003}. In \cite{les7} a formal characterization
    of the genealogy is provided by 
    extending this duality relation into the path-wise setting via a lookdown construction  \cite{DonnellyKurtz1999}. In this paper
    we introduce a novel class of self-similar (SS) population models whose time-changed frequency process are in duality with
    general $\Lambda$-coalescents. Given the extent of the current paper, we leave the development of lookdown constructions for general self-similar populations for future work, and only concentrate on describing duality relations in expectation that, by incorporating both the frequency process and the total population size, complement the work of \cite{BertoinLeGall2003}.

    The family of $\Lambda$-FV (resp. $\Lambda$-coalescents) generalizes the standard FV process (resp. the Kingman coalescent) to capture the population genetics dynamics of a wider range of neutral populations, including populations with highly skewed offspring distributions (but also many more models, including models with selection). 
    Still,  \cite{les7} 
    proved that {only the Beta} subfamily can be obtained from a branching process by using their method
    based on path-wise
    random time changes and dualities. For the rest of the super-processes (measure-valued branching processes), the associated frequency process is not Markov in itself under any random time change that is written in terms of the total population size (see their Lemma 3.5 for further details). This apparent lack of a branching process counterpart for the rest of the $\Lambda$ family has motivated research seeking variants of the 
    main results in \cite{les7} to obtain different $\Lambda$-FV (resp. $\Lambda$-coalescents) 
    for the frequency process (resp.  genealogy) of branching processes, but using different transformations. An example is the culling procedure in \cite{CaballeroCasanovaPerez2023} (see also \cite{CasanovaNunezPerez2023}) who work
    with the two-dimensional counterparts of FV processes, { and who approximate populations with constant size and obtain results in distribution, losing the path-wise quality. Another example is the work of \cite{JohnstonLambert2023}
    who study the genealogy of multi-type branching processes locally in time by computing 
    the corresponding coalescence rates, which notably depend on the size of the populations of each of the finitely-many types.} In our case it has motivated a change of perspective by now focusing on the self-similarity property, enabling us to use robust path-wise tools such as 
    the Lamperti transformation in the infinite-type setting. {This technique also allows us to incorporate
    dependencies in the reproduction dynamics of the individuals, mainly through the total population size, accommodating in a simple way the effect of this size on the evolution of the frequency process and the corresponding coalescent dual.}

Let us introduce self-similar measure-valued Markov processes and the related Markov additive property.
\begin{definition}[Self-similar process] 
A Markov process $\stp{\mu_t}$ taking values in the space of finite and positive measures $\sM(\typeS)$, and more generally {a  Markov process 
$\stp{X_t}$ taking values in a linear conic space},
is said to satisfy
the \emph{self-similarity} property with index {$\alpha\in\R$} (denoted $\alpha$-SS) if
\begin{equation}\label{eq:ssDef}
\forall a>0, \quad \law\(X_t \mid \P^X_x\) = \law\(aX_{a^{-\alpha}t} \mid \P^X_{a^{-1}x}\).
\end{equation}
\end{definition}
\begin{definition}[Markov additive processes (MAP)]
A two-coordinate process $\stp{A_t}=\stp{\rho_t,\xi_t}$ with lifetime $\abstime^A$ and taking values on $S\times\R$, 
with $S$ being {any measurable space, typically the unit sphere of a normed vector space},
is said to 
be a \emph{Markov additive process}
if for any $y\in S$, $z\in\R$, $s,t\geq0$, and for any positive measurable function $f$ on
$S\times\R$, one has
\begin{equation}\label{eqMAP}
\E^A_{{\theta},z}\[f(\rho_{t+s}, \xi_{t+s}-\xi_t), t+s<\abstime^A\given{\mathcal F_t}\] 
= \E^A_{\rho_t,0}\[f(\rho_s,\xi_s), s<\abstime^A\]\Ind{t<\abstime^A}.
\end{equation}
\end{definition}
   Recalling the notation  $\norm{\mu}= \mu(\typeS)$, and assuming that $\typeS$ is any compact and polish space, we obtain the following result: 
    there exists a standard measure-valued $\alpha$-SS Markov  process $\stp{\mu_t}$ with generator of the form
    \begin{align}\label{eq:generatormvssMp}
    \Fall_\alpha F(\mu)=& \frac{1}{\norm{\mu}^\alpha} \left( \Gdrift{\drift}F(\mu) + \Gdiff{\sigma}F(\mu) + \Gjump{\Lambda} F(\mu)\right) 
    \end{align}
    where, for $\sigma\geq 0, \drift\in\R$, and {$\Lambda\in\sM((0,1))$}, the above operators 
    are defined by
    \begin{align*} 
\Gdrift{\drift}F(\mu)&= \int_{\typeS} \mu(da) \drift F'(\mu; a),\\
\Gdiff{\sigma}F(\mu)&= \norm{\mu} \int_\typeS \mu(da)\frac{\sigma^2}{2} F''(\mu;a,a),\text{ and }\nonumber\\
\Gjump{\Lambda}F(\mu)&=\int_{\typeS} \frac{\mu(da)}{\norm{\mu}} \int_{{(0,1)}}\frac{\Lambda(d\zeta)}{\zeta^2} \bigg\{ F\left(\mu + \norm{\mu}\frac{\zeta}{1-\zeta} \delta_a\right)\\
      &\quad\quad- F(\mu) - \norm{\mu}(\abs{\log(1-\zeta)}\Ind{\zeta<1/2}) F'(\mu; a)\bigg\}.
\end{align*}
The state $\mu=0$ (the zero measure) is absorbing for this Markov process. Furthermore, letting $\abstime^\mu\coloneqq \inf\{t\geq 0\colon \norm{\mu_t}=0\}$ be the extinction time of the population and setting
$$ c_\alpha(t)=\inf\left\{s\geq 0\colon \int_0^s \norm{\mu_u}^{-\alpha} du\geq t\right\}, \quad t\in \left[0,\int_0^{\abstime^\mu}\norm{\mu_u}^{-\alpha} du\right];$$ 
the process $\stp{\frac{\mu_{c_\alpha(t)}}{\norm{\mu_{c_\alpha(t)}}},\log(\norm{\mu_{c_\alpha(t)}})}$ is a MAP. Its first coordinate is a $(\Lambda+\sigma^2\delta_0)$-Fleming-Viot process; whereas the second coordinate is {the Lévy process with Lévy–Khintchine characteristic triplet $(-\sigma+\drift, {\sigma}, \LevyMeasure)$. Here $\LevyMeasure(d\zeta)$ is the pushforward of the measure $\zeta^{-2}\Lambda(d\zeta)$  under the transformation $\zeta\to -\log(1-\zeta)$ on $(0,1)$}.
\par
Furthermore, the reproduction dynamics of $\stp{\mu_t}$ have the following interesting biological interpretation: the operator   
\begin{align}\label{eq:generatormvMp}
	\Gall\coloneqq \Fall_0\equiv \Gdrift{\drift} + \Gdiff{\sigma} + \Gjump{\Lambda}, 
\end{align}
is in fact the generator of a measure-valued Markov process $\stp{\nu_t}$ describing a population whose total size $\stp{\norm{\nu_t}}$ evolves as the exponential of a Lévy process (see Theorem \ref{th:MVPsmhExist}). 
Then, the extra scaling term  $\norm{\mu}^{-\alpha}$ appearing in \eqref{eq:generatormvssMp} can be interpreted as a regulator of the overall reproduction rate of the entire population as a function of the current total population size;
where $\alpha$ is a parameter specifying the 
strength of this modulation. The total size of the population $\stp{\norm{\mu_t}}$, however, does not reach a constant equilibrium but rather fluctuates stochastically over time as a positive $\alpha$-SS Markov process {on $[0,\infty)$} with non-negative jumps. Moreover, 
 $\stp{\mu_t}$ can be constructed via a self-similar Lamperti time change $\gamma_\alpha(t)$  of the process $\stp{\nu_t}$ ---see Theorems \ref{th:banachLamperti} and \ref{th:mvpLamperti} below for details---.  Since
$\stp{\norm{\nu_t}}$ evolves as the exponential of a Lévy process, the original work of Lamperti \cite{Lamperti1972} ensures that the total size
of the time-changed process $\stp{\norm{\mu_t}}=\stp{\norm{\nu_{\gamma_\alpha(t)}}}$ is in fact a positive self-similar Markov process. Conversely, $\stp{\nu_t}$ can be recovered back via  the ``reverse'' self-similar Lamperti time change $c_\alpha(t)$ of the process $\stp{\mu_t}$. \par
Looking back at  $\Gjump{\Lambda}$ in the generator \eqref{eq:generatormvMp} of the process $\stp{\nu_t}$
	we observe that it has jumps of the form
	$$
	\nu \to \nu + \norm{\nu}x\delta_a
	$$
	where a new atom is added. The location of the new atom $a\in\typeS$ is chosen according to the empirical distribution $\nu/\norm{\nu}$ before the jump. Importantly, its size $\norm{\nu}x$ depends on the total mass of the population $\norm{\nu}$, thus
	these processes will not be branching processes in general, nor will be their time-changed counterparts $\stp{\mu_t}$. 
	However, the frequency process $\stp{\nu_t/\stp{\norm{\nu_t}}}$ 
	is Markov with jumps of the form
	\begin{align*}
		\rho=\frac{\nu}{\norm{\nu}} &\to \frac{\nu}{\norm{\nu}} \left(\frac{1}{1+x}\right) + \frac{x}{1+x}\delta_a\\
		&= \rho  \left(\frac{1}{1+x}\right) + \frac{x}{1+x}\delta_a.
	\end{align*}
	The particular choice $x=\zeta / (1-\zeta)$ in $\Gjump{\Lambda}$ ensures that the resulting dynamics are exactly those of the $\Lambda$-FV processes (see e.g. \eqref{eq:FVgenerator}). Thus, after time-changing the process $\stp{\mu_t}$ by $c_\alpha(t)$ to recover such a process $\stp{\nu_t}$, and then renormalizing, one
	obtains a $\Lambda$-FV process.\par 
    By picking the right 
	combination of the parameters $\alpha$ and $\Lambda\equiv\Lambda_\alpha$ in \eqref{eq:generatormvssMp} for the dynamics of the process $\stp{\mu_t}$
	one can show, via a simple computation on the generator, 
	that the sources of dependency of the reproduction dynamics on the total mass $\norm{\mu}$ can be canceled out. The resulting process $\stp{\mu_t}$ is, in this case, a  stable branching process. Thus, our results recover the characterization in \cite{les7} of the frequency process of $\beta$-stable measure-valued branching process in the cases when $\beta$, the stability index therein, satisfies $\beta\in(1,2)$. Indeed, it can be seen that these processes enjoy the $(\beta-1)$-SS property, as seen through an extension of the result in \cite{KyprianouPardo2008} from $\beta$-stable $\R_+$-valued branching processes into the measure-valued setting. Time-changing  $\stp{\mu_t}$
	by the Lamperti time change $c_{\beta-1}(t)$ results in a process $\stp{\nu_t}$ whose frequencies evolve as a Beta-FV, see Remark \ref{rem:alfaStableCase} below.\par
	Finally, a key element in the construction of the process $\stp{\nu_t}$ is the characterization of a new duality relation (Theorem \ref{th:duality}) between $\stp{\nu_t}$ and a two-coordinate process $\stp{\Pi_t,Z_t}$. The latter consists of a $\Lambda$-coalescent $\stp{\Pi_t}$
	that is coupled with the exponential of a Lévy process $\stp{Z_t}$. By keeping the information of the total population size $\norm{\nu_t}$, this duality relation extends that described in \cite{BertoinLeGall2003} that only considers (normalized) FV processes and $\Lambda$-coalescents.
\par
Our method is based on a generalization of the Lamperti transformation \cite{Lamperti1972,ChaumontPantiRivero2013, AliliChaumont2017} into the infinite-dimensional
 setting, which we now describe (see section \ref{sec:BanachLamperti} for further details). Let $E$ be a conic subset
 of a normed vector space $(\V,\norm{\cdot})$, and suppose that 
 $(E,\dist)$ is a metric space such that the map $x\to\norm{x}$ is continuous.
 We also assume $(E,\dist)$
to be locally-compact and second-countable, and augment it to $\hat E=E\cup\{\pinfty\}$ 
where $\pinfty$ is a point at infinity if $E$ is not compact, or just an isolated point 
if $E$  is compact. In the application above, $E$ is the space of {finite positive measures $\sM(\typeS)$ over a compact and Polish type space $\typeS$ with at least countably  many elements (this is no true restriction since we can always add countably many elements to $\typeS$
	and still obtain a compact Polish space). The space $\sM(\typeS)$ is endowed with the total variation norm
which coincides with
$\norm{\mu}=\mu(\typeS)$ in this setting. However, the distance $\dist$ that we use on $\sM(\typeS)$ 
is any metrization of the weak topology on $\sM(\typeS)$.} \par 
We consider only standard \cadlag Markov process on $\hat E$
that are absorbed at $\{\pinfty,0\}$. 
For such a  process $X$ we write  
 $\abstime_\pinfty^X$ 
for its absorption time to $\pinfty$, $\abstime^X_0$ for its absorption time to $0\in E$, and define 
its lifetime $\abstime^X\coloneqq \abstime_\pinfty^X\wedge \abstime_0^X$.
The self-similar Lamperti transformation can be expressed as follows.

\begin{theorem}\label{th:banachLamperti}
Let $\stp{X_t}$ be a standard $\alpha$-SS Markov process. Consider the additive functional 
$t\to\int_0^t\norm{X_u}^{-\alpha}du$ for $t\in [0,\abstime^X]$, and its generalized inverse
$$
c_\alpha(t)\coloneqq \inf\left\{s>0\colon \int_0^s\norm{X_u}^{-\alpha}du\geq t\right\}, \quad t\in\left[0,\int_0^{\abstime^X}\norm{X_u}^{-\alpha}du\right].
$$
Then the process $\stp{A_t}=\stp{\frac{X_{c_\alpha(t)}}{\norm{X_{c_\alpha(t)}}},\log(\norm{X_{c_\alpha(t)}})}$ is a standard MAP 
with lifetime $\abstime^A
\aseq \int_0^{\abstime^X}\norm{X_u}^{-\alpha}du$.\par

Conversely,  let $\stp{A_t}=\stp{\rho_t,\xi_t}$  be  a standard MAP, 
and let 
{$\abstime^{A}=\abstime_\pinfty^{e^\xi}\wedge \abstime^{e^\xi}_0$} be its lifetime, after which it is absorbed
at some extra state.
For any $\alpha\geq 0$, consider the inverse additive functional
$$
\gamma_\alpha(t)\coloneqq \inf\lbr s>0\colon \int_0^s e^{\alpha \xi_u} du\geq t\rbr, \quad t\in\left[0,\int_0^{\abstime^{A}} e^{\alpha \xi_u} du\right].
$$ 
Then the process $\stp{X_t}=\stp{\rho_{\gamma_\alpha(t)}e^{\xi_{\gamma_\alpha(t)}}}$ is a standard  $\alpha$-SS Markov process with lifetime $\abstime^X=\abstime_\pinfty^X\wedge \abstime_0^X\aseq \int_0^{\abstime^{A}} e^{\alpha \xi_u} du$.
\end{theorem}
Our results show an unexplored link between the fields of mathematical population genetics and self-similar Markov processes in infinite dimensions, motivating new research and opening new questions in both of these fields separately, but also at their intersection.
 For instance, the populations driven by 
$\Fall_{\alpha}$ above, characterized by four parameters (the Lévy triplet and the self-similarity index)
 constitute only a sub-class of the entire family of measure-valued self-similar processes
which is yet to be fully described.  
Whether one will need a new model to describe their corresponding genealogies, outside or extending the family of $\Lambda$ (and more generally $\Xi$) coalescent processes, is still an open question. At the same time, measure-valued processes, together with the well-established analytic tools available in population genetics such as duality methods, may serve as a suitable template for the development of the theory of self-similar Markov processes and their Lamperti transformations 
in infinite dimensions.

{\bf Structure of the manuscript: }
{the rest of the manuscript is structured as follows. Section \ref{sec:BanachLamperti} is devoted to 
the development of our main methodological result, the Lamperti transformation in the infinite-dimensional setting (Theorem \ref{th:banachLamperti}). On the other hand, our main phenomenological results for self-similar populations are
presented in section \ref{sec:genealogySSMVps}, which is preceded by section \ref{sec:preliminaries} in which we describe
some necessary preliminaries on coalescent processes, FV processes, and Dawson-Watanabe processes. 
In turn, the proofs of our results of section \ref{sec:genealogySSMVps} are organized as follows: 
the results on the Lamperti  transformations of $\stp{\mu_t}$ and $\stp{\nu_t}$ are proved in section \ref{sec:MVLampertiTransformation}, making use of the general theory  developed  in section \ref{sec:BanachLamperti}. On the other hand, preparing for the construction of the process $\stp{\nu_t}$, 
in section \ref{sec:technicalAndGenerators} we provide moment bounds for the exponential of a Lévy process,
as well as regularity results on the generators of $\stp{\nu_t}$ and
its dual $\stp{\Pi_t,Z_t}$. These technical results are used in section \ref{sec:duality} to prove their duality relation, which is a key ingredient, in sections \ref{sec:nucons} and \ref{sec:dualConstruction}, for the formal construction
of the processes $\stp{\nu_t}$ and $\stp{\Pi,Z_t}$ via martingale problem characterizations. }
\par

{\bf Notation: }
we let $\Rp=[0,\infty)$.
The symbol $\Bo(\cdot)$ stands for the Borel $\sigma$-algebra in any topological space.  Also $\cf(\cdot)$ (resp. $\bof(\cdot)$) will refer to the space of $\R$-valued continuous (resp. measurable) 
functions defined on some topological (resp. measurable) space; whereas
$\cbf(\cdot)$ (resp. $\bbf(\cdot)$) refers to its bounded counterpart. The symbol $\czdf{k}$ refers to
the space of continuous functions vanishing at infinity and that have a continuous $k$-th derivative, and $\cKdf{k}$ refers to the
subspace of functions with compact support.\par 
We will also write $f(x)=\bO(g(x))$ as $x\to x_0$ whenever $f(x)/g(x) \toas{x\to x_0}C\in \Rp$, and
$f(x)=\lo(g(x))$ as $x\to x_0$ whenever $f(x)/g(x) \toas{x\to x_0}0.$ Also we will write that $f_n$ converges to $f$ boundedly point-wise if $\sup_n \norm{f_n}_\infty<\infty$ and $f_n\toas{n\to\infty} f$ point-wise.\par
Also $\sM(\typeS)$ (resp. $\sPM(\typeS)$) refers to the space of finite positive (resp. probability) measures on the polish and compact space $\typeS$.\par
Additionally, we will denote by $(\A,D)$ an operator defined
on a set of functions $D$, and also write $(\A_1,D_1)\subset(\A_2,D_2)$ to mean that $D_1\subset D_2$ and $\A_2=\A_1$ on $D_1$. Also, $\D(\A)$  refers to the domain of a closed operator $\A$.

\section{The Lamperti transformation in normed spaces}\label{sec:BanachLamperti}
In this section we get inspiration from the works of \cite{Lamperti1972, ChaumontPantiRivero2013, AliliChaumont2017} and generalize the self-similar Lamperti transformation to processes taking values in a conic subset
$E$ of a normed space $(\V,\norm{\cdot})$.
Some of the arguments used in our main proofs  can  be found
in \cite{Lamperti1972, AliliChaumont2017, HassMiermont2011}, we expand them to general state spaces and also deal with possible
explosion of the processes involved. 
We assume that all the processes
that we consider are standard in the sense of Definition 9.2 in \cite{GetoorBlumenthal1968}; namely that
they satisfy the following:
\begin{enumerate}
    \item their respective filtrations $(\mathcal F_t)_{t\ge0}$ are right-continuous,
       \item they are absorbed at $\pinfty$ at time  $\abstime_\pinfty\in [0,\infty]$,
    \item they are \cadlag and quasi-left-continuous on $[0,\abstime_\pinfty)$,
    \item they are strong-Markov processes with measurable probability kernels 
    $P_t(x,\cdot)$ from $\hat E$ to $\hat E$ corresponding to the law of the process
    at time $t$.
\end{enumerate}\par
Let us call $D([0,\infty),\hat E)$ the trajectory space of such processes, i.e. the
space of trajectories on $\hat E$ that are \cadlag on $[0,\abstime_\pinfty)$ and have constant
value $\pinfty$ on $[\abstime_\pinfty,\infty)$. 
We endow $D([0,\infty),\hat E)\subset {\hat E}^{\Rp}$ with $\mathcal F$, the trace $\sigma$-algebra 
induced by $\sigma(\pi_t; t\geq0)$ 
where $\pi_t$ is the projection at time $t$. Note that when $E$ is a locally-compact second-countable metric space, the extended space
$\hat E$ 
becomes compact and metrizable (see e.g. Proposition VII.1.15 in \cite{Cohn2013}) and
$\mathcal F$ coincides with the Borel $\sigma$-algebra induced by the Skorohod topology on $D([0,\infty),\hat E)$
(Theorem 12.5 in \cite{Billingsley99}). Also we will denote by $\P_x(\cdot)$ the law  on $\hat E^{\Rp}$ of the processes started at $x$. We also recall the notation $\abstime^X\coloneqq \abstime_\pinfty^X\wedge \abstime_0^X$ for the lifetime of such a process $X$.
\par

Recall that an $\hat E$-valued Markov process $(X_t)_{t\ge0}$ 
is said to satisfy
the {self-similarity (SS)} property with index  {$\alpha\in \R$} ($\alpha$-SS) if \eqref{eq:ssDef} holds.
Alternatively, it is easily seen that
the process $X$ is $\alpha$-SS if and only if
for all $t\geq0,a>0,x\in E, B\in\Bo(E)$, {its transition kernels $P^X_t$ satisfy}
\begin{equation*}
P^X_t(x,B) = P^X_{a^{-\alpha}t}\left(a^{-1}x, a^{-1}B\right).
\end{equation*}

\begin{definition}[Scalar multiplicative homogeneous (SMH)]
An $\hat E$-valued Markov process $Y$ 
is said to be  \emph{scalar multiplicative homogeneous (SMH)} if its transition
kernels $P^Y_t$ satisfy, for all $t\geq0,u>0,x\in E, A\in\Bo(E)$,
\begin{equation}\label{eq:kMH}
 P^Y_t(x,A) = P^Y_t(ux,uA),
\end{equation}
or, in other words, if
$$
\forall u>0, \quad \law\(Y_t\mid\P_x^Y\) = \law\(uY_{t}\mid \P_{u^{-1}x}^Y\).
$$
Observe that this property is exactly the self-similarity property of index $\alpha=0$.
\end{definition}

Recall the Markov addtivie property introduced in \eqref{eqMAP}. It can be seen that this is equivalent to the following:
the process $\stp{\rho_t,\xi_t}$ is a MAP if
and only if,
for all $t\geq 0, a\in\R, (\theta,z)\in S\times\R^+$ and $B\in\Bo(S\times\R^+)$,
\begin{equation}\label{eq:kMAP}
P^{A}_t((\theta,z), B) = P^{A}_t((\theta,z+a), B + (0,a));
\end{equation}
where we have written $B+(0,a)\coloneqq \{(\rho,z+a) \colon (\rho,z)\in B\}$. 
This can be interpreted as saying that $\stp{A_t}=\stp{\rho_t,\xi_t}$ is additive-homogeneous on the second coordinate. 
Also we define its lifetime $\abstime^{A}\coloneqq\abstime^\xi_\pinfty$ (the explosion time of $\stp{\xi_t}$ on the one-point compactification of $\R$, after which we assume that $\stp{A_t}$ is absorbed at some extra state).
We refer the interested reader to \cite{KyprianouPardo2022book} for a thorough 
exposition of the subject for $\R^d\times\R$-valued MAPs.

We will establish transformations between SS and SMH processes on the one hand, and
between SMH processes and MAPs on the other. In fact, the latter is simply a bijection given by the ``log-polar decomposition'' isomorphism
$\Phi\colon E\setminus\{0\} \to S\times\R$ defined as 
$\Phi(x)=\left(x/\norm{x}, \log(\norm{x})\right)$. We have the following.

\begin{proposition}[SMH $\iff$ MAP]\label{prop:SMHtoMAP}
Let $Y$ be a SMH Markov process with trajectories in $D([0,\infty),\hat E)$ and absorbed at $0$,
and set $\abstime^Y= \abstime^Y_0 \wedge \abstime^Y_\pinfty$.
Then $\stp{A_t}=\stp{\Phi(Y_t)}$ 
is a MAP with lifetime $\abstime^A=\abstime^Y$. \par
Conversely, 
if $\stp{A_t}=\stp{\rho_t,\xi_t}$ is a MAP with lifetime $\abstime^A= \abstime^{\xi}_\pinfty$, then $Y=\stp{\Phi^{-1}(A_t)}\equiv\stp{e^{\xi_t}\rho_t}$ is a 
SMH Markov process with lifetime $\abstime^Y=\abstime^A$.
\end{proposition}
\begin{proof}
Given that $\Phi$ is
bijective and continuous except at the absorbing state $0$, it is clear that the transformed processes are standard whenever the starting process is. Thus we need only verify \eqref{eq:kMAP}
in the first case, and \eqref{eq:kMH} in the second. We only do this for the first. We have, assuming \eqref{eq:kMH} in the
second equality below,
and for all $t\geq 0, a\in\R, (\theta,z)\in S\times\R,B\in\Bo(S\times\R)$,
\begin{align*}
P^A_t((\theta,z), B) &= P^Y_t( e^{z}\theta, {\Phi^{-1}(B)}) =  P^Y_t(e^{z+a}\theta, e^{a}{\Phi^{-1}(B)})\\
&= P^A_t((\theta,z + a), B + (0,a)). 
\end{align*} 
\end{proof}
The transformation between SS and SMH processes is given in terms of random time changes. 
The
proof follows the heuristics in the proof of Theorem 2.3 in \cite{AliliChaumont2017} but adapted to our setting.

\begin{theorem}[Self-Similar Lamperti Time Change]\label{th:smhssbijection}
Let $\stp{X_t}$ be a  standard Markov process with trajectories
in $D([0,\infty),\hat E)$. Let 
 $\abstime_\pinfty^X$ 
be its explosion time, $\abstime^X_0$ be its absorption time to $0\in E$, and define 
its lifetime $\abstime^X= \abstime_\pinfty^X\wedge \abstime_0^X$.
Consider the additive functional 
$t\to\int_0^t\norm{X_u}^{-\alpha}du$ for $t\in [0,\abstime^X]$, and its generalized inverse
\begin{align*}
c_\alpha(t)&\coloneqq \inf\left\{s>0\colon \int_0^s\norm{X_u}^{-\alpha}du\geq t\right\}, & t\in\left[0,\int_0^{\abstime^X}\norm{X_u}^{-\alpha}du\right],
\\
c_\alpha(t)&\coloneqq \abstime^X, & t\in\left[\int_0^{\abstime^X}\norm{X_u}^{-\alpha}du, \infty\right).
\end{align*}
\begin{enumerate}
\item \label{thIt:sstomh} 
If $\stp{X_t}$ is $\alpha$-SS, then the process $\stp{Y_t}=\stp{X_{c_\alpha(t)}}$ is a standard SMH Markov process with lifetime $\abstime^Y=\abstime_\pinfty^Y\wedge \abstime_0^Y\aseq \int_0^{\abstime^X}\norm{X_u}^{-\alpha}du$.
\item \label{thIt:sstomhEq} 
We have
$$
\abstime^X \aseq S_Y \coloneqq \sup\left\{\int_0^t \norm{Y_u}^\alpha du \colon \int_0^t \norm{Y_u}^\alpha du <\infty \right\}.
$$
Also the additive functional $c_\alpha(t)$ satisfies
\begin{equation}\label{eq:c_alphaEquation}
	\stp{c_\alpha(t)} \aseq \stp{\int_0^t \norm{Y_s}^\alpha ds\wedge S_Y}.
\end{equation}

Furthermore, 
{setting $X_\pinfty=\pinfty$ if $\alpha>0$, and $X_\pinfty=0$ if $\alpha\leq 0$},
the process $\stp{Y_t}$ is the unique solution to
\begin{equation}\label{eq:SStoSMHtrajEq}
\stp{Y_t}
=\stp{X_{\int_0^t \norm{Y_s}^\alpha ds\wedge S_Y}}.
\end{equation}
\end{enumerate}

Conversely, let $\stp{Y_t}$ be a standard Markov process with trajectories
in $D([0,\infty),\hat E)$.
Let  
$\abstime^Y= \abstime_\pinfty^Y\wedge \abstime^Y_0$ be its lifetime.
Consider, the inverse additive functional
\begin{align*}
\gamma_\alpha(t)&\coloneqq \inf\lbr s>0\colon \int_0^s \norm{Y_u}^\alpha du\geq t\rbr, & t\in\left[0,\int_0^{\abstime^Y} \norm{Y_u}^\alpha du\right]\\
\gamma_\alpha(t)&\coloneqq  \abstime^Y, &  t\in\left[\int_0^{\abstime^Y} \norm{Y_u}^\alpha du,\infty\right).
\end{align*}
\begin{enumerate}[resume]
\item \label{thIt:mhtoss} If $\stp{Y_t}$ is SMH, then the process $\stp{X_t}=\stp{Y_{\gamma_\alpha(t)}}$ is a standard  $\alpha$-SS Markov process with lifetime $\abstime^X=\abstime_\pinfty^X\wedge \abstime_0^X\aseq \int_0^{\abstime^Y} \norm{Y_u}^\alpha du$.
\item \label{thIt:mhtossEq}
We have
$$
\abstime^Y \aseq S_X \coloneqq \sup\left\{\int_0^t \norm{X_u}^{-\alpha} du \colon \int_0^t \norm{X_u}^{-\alpha} du <\infty \right\}.
$$
Also the additive functional $\gamma_\alpha(t)$ satisfies
$$
\stp{\gamma_\alpha(t)} \aseq \stp{\int_0^t \norm{X_s}^{-\alpha} ds\wedge S_X}.
$$
Furthermore, 
{setting $Y_\pinfty=0$ if $\alpha>0$, and $Y_\pinfty=\pinfty$ if $\alpha\leq 0$},
the process $\stp{X_t}$ is the unique solution to
\begin{equation}\label{eq:SMHtoSStrajEq}
\stp{ X_t}
	=\stp{Y_{\int_0^t \norm{X_s}^{-\alpha} ds\wedge S_X}}.
\end{equation}
\end{enumerate}
\end{theorem}

\begin{proof}
{To ease notations we will omit the index $\alpha$ in $c_\alpha(t)$ and $\gamma_\alpha(t)$.}

By  Propositions IV.1.6 and IV.1.13 in \cite{GetoorBlumenthal1968} (and the discussion around eq. IV.1.8 therein), the mappings $t\to\int_0^t \norm{X_u}^{-\alpha}du$
and $t\to \int_0^t \norm{Y_u}^\alpha du$ 
both define continuous strong additive functionals  of
$(\stp{X_t},\abstime^X)$
and $(\stp{Y_t}, \abstime^Y)$ respectively (see Definitions IV.1.1 and IV.1.11 in \cite{GetoorBlumenthal1968}). Then,
by Exercise V.2.11 iv) in \cite{GetoorBlumenthal1968}, the time-changed processes $\stp{X_{c(t)}}$ and $\stp{Y_{\gamma(t)}}$
are strong Markov processes with lifetimes $\int_0^{\abstime^X}\norm{X_u}^{-\alpha}du$ and $\int_0^{\abstime^Y}\norm{Y_u}^\alpha du$ in each case. Since  $c(t)$ (resp. $\gamma(t)$)
is continuous on $\left[0,\int_0^{\abstime^X}\norm{X_u}^{-\alpha}du\right)$ (resp. $\left[0,\int_0^{\abstime^Y}\norm{Y_u}^\alpha du\right)$), 
it follows that $\stp{X_{c(t)}}$ (resp. $\stp{Y_{\gamma(t)}}$) is quasi-left continuous. Also, by Lemma \ref{le:TimeChangeDMaeasurable} below,
the mapping $x\to \E_x^X[ f(X_{c(t)})]$ is measurable for every $t\geq 0$ and $f\in \bbf(\hat E)$, the space of bounded Borel functions on $\hat E$, so that the process defined by $Y_t=X_{c(t)}$ in $\ref{thIt:sstomh}$ 
is indeed a standard process. By an analogous argument, the process defined by $X_t=Y_{\gamma(t)}$ in $\ref{thIt:mhtoss}$ is also a standard process.

We now show that the process $\stp{Y_t}$ in $\ref{thIt:sstomh}$ is SMH.
Let $\hat c(t)$ be the functional $c(t)$ applied to the process $\stp{\hat X_t} \coloneqq \stp{aX_{a^{-\alpha}t}}$. Observe that the change of variable $v=a^{-\alpha}u$ yields
\begin{align*}
a^{-\alpha}\hat c(t)&=a^{-\alpha}\inf \lbr s\geq 0\colon \int_0^s \norm{aX_{a^{-\alpha}u}}^{-\alpha} du\geq t \rbr\\
&=\inf \lbr a^{-\alpha}s\geq 0\colon \int_0^{a^{-\alpha} s} \norm{X_{v}}^{-\alpha} dv\geq t \rbr \\
&=\inf \lbr s\geq 0\colon \int_0^{ s} \norm{X_{v}}^{-\alpha} dv\geq t \rbr\\
&=c(t).
\end{align*}
Thus $\hat X_{\hat c(t)}=aX_{a^{-\alpha}\hat c(t)}=aX_{c(t)}$. 
This, together wit the $\alpha$-SS property of $\stp{X}$ give, for $B\in  \Bo(\hat E), x\in \hat E$, and $a>0$,
\begin{align*}\label{eq:sstomh1}
     \P_x^Y( Y_{t} \in B)&= \P_x^X(X_{c(t)} \in B) 
    = \P_{a^{-1}x}^X(\hat X_{\hat c(t)} \in B) 
    = \P_{a^{-1}x}^X(a X_{c(t)} \in B)\\
    &= \P_{a^{-1}x}^Y( Y_{t} \in a^{-1}B),
\end{align*}
so that $\stp{Y_t}$ is SMH.\par
We proceed similarly to prove that the process $\stp{X}$ in $\ref{thIt:mhtoss}$ is self-similar.
For $a>0$ let $\hat \gamma(t)$ be the functional $\gamma(t)$ applied to the process $\stp{\hat Y_t}\coloneqq \stp{aY_t}$.
Observe that
\begin{align*}
    \hat \gamma(t) &=
    \inf \lbr s\geq 0 \colon \int_0^s \norm{aY_u}^\alpha du\geq t  \rbr\\
    &=\inf \lbr s\geq 0 \colon \int_0^s \norm{Y_u}^\alpha du\geq a^{-\alpha}t  \rbr\\
    &=\gamma(a^{-\alpha}t).
\end{align*}
The above yields $\hat Y_{\hat \gamma(t)}=aY_{\gamma(a^{-\alpha}t)}$.
This, together with the SMH property of $\stp{Y}$ in the second equality below give
\begin{align*}
\P_x^X(X_t \in B) &= \P_{x}^Y(Y_{\gamma(t)}\in B)
=\P_{a^{-1}x}^Y(\hat Y_{\hat \gamma(t)}\in B)
=\P_{a^{-1}x}^Y(aY_{\gamma(a^{-\alpha}t)}\in B)\\
&=\P_{a^{-1}x}^X(aX_{a^{-\alpha}t} \in B),
\end{align*}
so that $\stp{X_t}$ is $\alpha$-SS.\par
We now show $\ref{thIt:sstomhEq}$.
Write $\psi(t)\coloneqq \int_0^t \norm{X_s}^{-\alpha}ds$.
Note that the function $\psi$ is a.s. strictly increasing and absolutely continuous on compact
time intervals contained in $[0,\abstime^X]$, with inverse function $c(t)$, and a.e. derivative $\norm{X_t}^{-\alpha}$. 
It follows that for any 
$\psi(ds)$-integrable function $\beta_s$ on $[0,\abstime^X]$, we have
\begin{equation*}
\int_0^t \beta_s \norm{X_s}^{-\alpha}ds=\int_0^{\psi(t)}\beta_{c(s)}ds.
\end{equation*}
If $t<\psi(\abstime^X)$, in particular if $\psi(\abstime^X)=\infty$,  we have $0\underset{a.e.}{<}\norm{X_s}^{\alpha}\underset{a.e.}{<}\infty$ on $s\in[0,c(t)]$  
and $c(t)<\abstime^X\leq \infty$. Then
\begin{align*}
c(t)
&= \int_0^{c(t)} \norm{X_s}^{\alpha} \norm{X_s}^{-\alpha} ds
= \int_0^{t} \norm{X_{c(s)}}^{\alpha}ds.
\end{align*}
Furthermore
\begin{align*}
S_Y=\lim_{u\uparrow \psi(\abstime^X)} \int_0^{u} \norm{X_{c(s)}}^{\alpha}ds
=\lim_{u\uparrow \psi(\abstime^X)} c(u)
\equiv\abstime^X.
\end{align*}
Thus
$$
c(t)\equiv c(t)\wedge \abstime^X=\int_0^t \norm{Y_{s}}^{\alpha}ds\wedge S_Y.
$$
In particular, recalling that we have set $X_\pinfty=\pinfty$ if $\alpha>0$ and $X_\pinfty=0$ if $\alpha\leq 0$,
we have
$$Y_{t}= X_{c(t)\wedge \abstime^X}=X_{\int_0^{t} \norm{X_{c(s)}}^{\alpha}ds \wedge\abstime^X}=X_{\int_0^t \norm{Y_{s}}^{\alpha}ds\wedge \abstime^X}.$$
To see that this is the unique solution note that if $\stp{Y'_t}$ satisfies \eqref{eq:SStoSMHtrajEq}, then the function
$$
\tilde c(t) = \int_0^t \norm{Y'_s}^{\alpha}ds = \int_0^t \norm{X_{\tilde c(s)}}^{\alpha} ds
$$
is continuous, strictly increasing, and has derivative 
$\norm{X_{\tilde c(s)}}^\alpha$ a.e. on $[0,t]$ whenever
$\tilde c(t)< \abstime^X$. Then, in this
case we have
$$
t=\int_0^t \norm{X_{\tilde c(s)}}^{\alpha}\norm{X_{\tilde c(s)}}^{-\alpha}ds=\int_0^{\tilde c(t)} \norm{X_s}^{-\alpha}ds
$$
which implies 
$\tilde c(t)=c(t)$ whenever $\tilde c(t)<\abstime^X$ and
$\tilde c(\psi(\abstime^X))=\lim_{t\uparrow \psi(\abstime^X)}c(t) = \abstime^X$. 
Since $Y'_t = X_{\tilde c(t) \wedge \abstime^X}$, we conclude
$Y'_t= X_{c(t)\wedge \abstime^X}  =  Y_t$ for all $t\geq0$.\par

Finally, we prove $\ref{thIt:mhtossEq}$ using analogous arguments.
Write $\varphi(t)=\int_0^t\norm{Y_s}^\alpha ds$.
Note that the function $\phi$ is a.s. strictly increasing and absolutely continuous on compact
time intervals contained in $[0,\abstime^Y]$, with inverse function $\gamma(t)$, and a.e. derivative $\norm{Y_t}^{\alpha}$. 
It follows that for any 
$\phi(ds)$-integrable function $\beta_s$ on $[0,\abstime^Y]$, we have
\begin{equation*}
	\int_0^t \beta_s \norm{Y_s}^{\alpha}ds=\int_0^{\phi(t)}\beta_{\gamma(s)}ds.
\end{equation*}
If $t<\phi(\abstime^Y)$, in particular if $\phi(\abstime^Y)=\infty$,  we have $0\underset{a.e.}{<}\norm{Y_s}^{\alpha}\underset{a.e.}{<}\infty$ on $s\in[0,\gamma(t)]$  
and $\gamma(t)<\abstime^Y\leq \infty$. Then
\begin{align*}
	\gamma(t)
	&= \int_0^{\gamma(t)}  \norm{Y_s}^{-\alpha} \norm{Y_s}^{\alpha}ds
	= \int_0^{t} \norm{Y_{\gamma(s)}}^{-\alpha}ds.
\end{align*}
Furthermore
\begin{align*}
	S_X=\lim_{u\uparrow \phi(\abstime^Y)} \int_0^{u} \norm{Y_{\gamma(s)}}^{-\alpha}ds
	=\lim_{u\uparrow \phi(\abstime^Y)} \gamma(u)
	\equiv\abstime^Y.
\end{align*}
Thus
$$
\gamma(t)\equiv \gamma(t)\wedge \abstime^Y=\int_0^t \norm{X_{s}}^{-\alpha}ds\wedge S_X.
$$
In particular, recalling that we have set $Y_\pinfty=0$ if $\alpha>0$ and $Y_\pinfty=\pinfty$ if $\alpha\leq 0$,
we have
$$X_{t}= Y_{\gamma(t)\wedge \abstime^Y}=Y_{\int_0^{t} \norm{X_{\gamma(s)}}^{-\alpha}ds \wedge S^X}.$$
To see that this is the unique solution note that if $X'_t$ satisfies \eqref{eq:SMHtoSStrajEq}, then the function
$$
\tilde \gamma(t) = \int_0^t \norm{X'_s}^{-\alpha}ds = \int_0^t \norm{Y_{\tilde \gamma(s)}}^{-\alpha} ds
$$
is continuous, strictly increasing, and has derivative 
$\norm{Y_{\tilde \gamma(s)}}^{-\alpha}$ a.e. on $[0,t]$ whenever
$\tilde \gamma(t)< \abstime^Y$. Then, in this
case we have
$$
t=\int_0^t \norm{Y_{\tilde \gamma(s)}}^{\alpha}\norm{Y_{\tilde \gamma(s)}}^{-\alpha}ds=\int_0^{\tilde \gamma(t)} \norm{Y_s}^{\alpha}ds
$$
which implies 
$\tilde \gamma(t)=\gamma(t)$ whenever $\tilde \gamma(t)<\abstime^X$ and
$\tilde \gamma(\phi(\abstime^Y))=\lim_{t\uparrow \phi(\abstime^Y)}\gamma(t) = \abstime^Y$. 
Then since $X'_t = Y_{\tilde \gamma(t) \wedge \abstime^Y}$, we conclude
$X'_t= Y_{\gamma(t)\wedge \abstime^Y}  =  X_t$ for all $t\geq0$.
\end{proof}

 By \eqref{eq:c_alphaEquation} the
time change $\gamma_\alpha$ is the right inverse of $c_\alpha$; i.e. $c_\alpha\circ \gamma_\alpha (t)\equiv t.$
\begin{corollary}[Lamperti time change bijection]
Let $\stp{Y_t}=\stp{X_{c_\alpha(t)}}$ be as in Theorem \ref{th:smhssbijection} \ref{thIt:sstomh}, and consider the time change $\gamma_\alpha$ of $\stp{Y_t}$. Then $\stp{X_t}\aseq \stp{Y_{\gamma_\alpha(t)}}.$
		\label{Th:smhtossTimeChangeBijection} 
		Conversely, let  $\stp{X_t}=\stp{Y_{\gamma_\alpha(t)}}$ be a is Theorem \ref{th:smhssbijection} \ref{thIt:mhtoss}, and consider the time change $c_\alpha$ of $\stp{X_t}$. Then $\stp{Y_t}\aseq \stp{X_{c_\alpha(t)}}.$
\end{corollary}

    The composition of the above two transformations
    between MAP and SMH processes, and between SMH and $\alpha$-SS processes respectively, leads to
    Theorem \ref{th:banachLamperti} which for $\R^d$-valued processes is the Lamperti transformation of \cite{AliliChaumont2017}; and
    for $\R^+$-valued processes is the original transformation of \cite{Lamperti1972}. We also refer the reader to the results in \cite{ChaumontPantiRivero2013} and  \cite{CasanovaMiroSchertzerJegousse2024}.

The following proposition gives a characterization of {$\alpha$-SS and} SMH processes in terms of their generators. For $b\geq 0$, let $\SSS_b$ be the operator that scales space by a factor of $b$, i.e. that takes $f\in\bbf(E)$ to $\SSS_b f(x)=f(bx)$.  Recall that we denote by $(\A,D)$ an operator defined
on a set $D$, and also write $(\A_1,D_1)\subset(\A_2,D_2)$ to mean that $D_1\subset D_2$ and $\A_2=\A_1$ on $D_1$. Also, $\D(A)$  refers to the full domain of a closed operator $\A$. 

\begin{proposition}\label{prop:smhssGenerators}
    Consider an operator $(\A,D_\A)$ satisfying $\SSS_b D_\A \subset D_\A$ for all $b\geq0$.
    Assume also that the solutions to the martingale problem for $(\A,D_\A)$ are unique. 
    \par
    Let $\stp{X_t}$ be a Markov process with generator $(\A,\D(\A))$ satisfying $(\A,D_\A)\subset (\A,\D(\A))$.
    Then the following are equivalent:
    \begin{enumerate}
        \item \label{itTh:Xss} The process $X$ is $\alpha$-SS.
        \item \label{itTh:ssGenerator} For all $b\geq0$ and $f\in D_\A$, $\A f  = b^{-\alpha}\SSS_{b^{-1}}\A\SSS_b f$.
    \end{enumerate}
    \par
    Similarly, let $\stp{Y_t}$ be a Markov process with generator $(\A,\D(\A))$ satisfying $(\A,D_\A)\subset (\A,\D(\A)).$
    Then the following are equivalent:
    \begin{enumerate}[resume]
        \item \label{itTh:Ysmh} The process $Y$ is SMH.
        \item \label{itTh:smhGenerator}For all $b\geq 0$ and $f\in D_\A$,  $\A  = \SSS_{b^{-1}}\A\SSS_b$, 
    \end{enumerate} 
\end{proposition}
\begin{proof}
We only prove the equivalence between $\ref{itTh:Xss}$ and $\ref{itTh:ssGenerator}$ since the
SMH case corresponds to $\alpha=0$.
Assuming $\ref{itTh:Xss}$ we obtain $\E^X_{x}[f(X_t)]=\E^X_{b^{-1}x}\left[\SSS_bf(X_{b^{-\alpha}t})\right]$ for every bounded function $f$. Taking time derivatives,
        \begin{align*}
\forall f\in D_\A,\quad \A f(x)&= \frac{d}{dt}\big\vert_{t=0} \E^X_{x}[f(X_t)]= \frac{d}{dt}\big\vert_{t=0}\E^X_{b^{-1}x}[\SSS_b f(X_{b^{-\alpha} t})]\\&=b^{-\alpha}\A \SSS_b f(b^{-1}x) = b^{-\alpha}\SSS_{b^{-1}}\A \SSS_b f(x),
    \end{align*}
and $\ref{itTh:ssGenerator}$ follows.
  
    For the converse note that we can compute the generator of the Markov process  $\{bX_{b^{-\alpha} t}; \P^X_{b^{-1}x}\}$ on $D_\A$ 
    in terms of $\SSS_b\A$; we obtain
    $$
    \frac{d}{dt}\big\vert_{t=0}\E^X_{b^{-1}x}[f(bX_{b^{-\alpha} t})]=b^{-\alpha}\lim_{t\downarrow 0}
    \frac{\E^X_{b^{-1}x}[\SSS_b f(X_{b^{-\alpha} t})]- \SSS_b f(b^{-1}x)}{b^{-\alpha}t}
    =b^{-\alpha}\A \SSS_b f(b^{-1}x).
    $$
    Then by $\ref{itTh:ssGenerator}$ the process $\{bX_{b^{-\alpha} t}; \P^X_{b^{-1}x}\}$
    solves the martingale problem for $(\A,D_\A)$ and $\ref{itTh:Xss}$ follows by uniqueness of solutions.  
    \end{proof}

We end this section with the following technical lemma that was used in the proof of Theorem \ref{th:smhssbijection}.
\begin{lemma}\label{le:TimeChangeDMaeasurable}
\begin{enumerate}We have the following measurability of mappings.
    \item \label{itLe:TimeChangeFunctDMeasurable} The mappings $\stp{x_t} \to \stp{c(t)}$ and $\stp{y_t} \to \stp{\gamma(t)}$ are measurable from $D([0,\infty),\hat E)$ to $D^0([0,\infty), [0,\infty])$,  the non-decreasing elements of $D([0,\infty),[0,\infty])$, endowed with the relative $\sigma$-algebra inherited from the Skorohod $\sigma$-algebra in $D([0,\infty),[0,\infty])$.
    \item \label{itLe:TimeChangeProcessDMeasurable} The mappings $\stp{x_t} \to  \stp{x_{c(t)}}$ and $\stp{z_t} \to \stp{z_{\gamma(t)}}$ are measurable from $D([0,\infty),\hat E)$ to $D([0,\infty),\hat E)$.
\end{enumerate}
\end{lemma}
\begin{proof}
Recall that, $E$ being locally compact and second countable, the Borel $\sigma$-algebra on $D([0,\infty),\hat E)$ generated by the Skorohod topology, and the trace $\sigma$-algebra generated by the 
finite-dimensional projections on $D([0,\infty),\hat E)$, coincide. Then, $\ref{itLe:TimeChangeFunctDMeasurable}$
follows from the fact that, for each $t_0$, the map $\stp{x_t} \to  c(t_0)$ is 
measurable from $\Bo(D([0,\infty),\hat E))$ to $\Bo(\hat [0,\infty])$, where we recall that $\Bo(\cdot)$ stands for the Borel $\sigma$-algebra in any topological space. Since $\Bo(D([0,\infty),[0,\infty]))$ is generated by the finite dimensional
projections (see section 12 in \cite{Billingsley99}), the latter implies that 
the map $\stp{x_t} \to \stp{c(t)}$ is measurable from $D([0,\infty),\hat E)$ to $D([0,\infty),[0,\infty])$.\par
On the other hand, $\ref{itLe:TimeChangeProcessDMeasurable}$ follows from $\ref{itLe:TimeChangeFunctDMeasurable}$ together with Appendix M16 in \cite{Billingsley99}.
\end{proof}

\section{Preliminary objects of study}\label{sec:preliminaries}
 
\subsection{$\Lambda$-coalescents}\label{sec:coalescents}
We expose the construction of coalescents with multiple merger from the seminal works of \cite{Pitman99,Sagitov99}.
For a positive integer  $p$, let $[p]=\{1,\cdots,p\}$ and $\PS{[p]}$ be the space of partitions of $[p]$ endowed with the discrete topology.
We call the elements of any partition $\pi\in\PS{[p]}$ the blocks of $\pi$ and denote its number by $\#\pi$. 
Let $\Lambda$ be a finite measure on $[0,1]$ which can be decomposed as
$$
\Lambda= \Lambda(\{0\})\delta_0 + \Ind{{(0,1)}}\Lambda.
$$
The $(p,\Lambda)$-coalescent process $\stp{\Pi_t}$ is a Markov jump process with values in $\PS{[p]}$ that evolves through
\say{coagulations} or mergers. The latter consists of constructing a new coarser partition of $[p]$ from 
an initial $\pi\in \PS{[p]}$
by taking the union of a collection of blocks  that are present in $\pi$. 
The coagulations of $\stp{\Pi_t}$ are directed by the measure $\Lambda$ via the
following rules; at time $t\geq0$:
\begin{description}
    \item [Pairwise coagulations:] Any pair of blocks of $\Pi_t$ coagulate at rate $\Lambda(\{0\})$.
    \item [Coin-flip coagulations:] Any collection of $2\leq i\leq j =\card{\Pi_t}$ blocks of $\Pi_t$, coagulate
    into a single block at rate $\beta_{j,i}^{(\Lambda)}\coloneqq \int_{{(0,1)}} \zeta^{i-2}(1-\zeta)^{j-i}\Lambda(d\zeta)$. 
\end{description}
The first dynamics correspond to those of  Kingman's coalescent \cite{Kingman1982}.
The second dynamics  have the following well-known interpretation: at rate $\zeta^{-2}{\Lambda(d\zeta)}$
a value $\zeta\in {(0,1)}$ is drawn; then, each block of $\Pi_t$ decides to participate in the coagulation event
with probability $\zeta$.
This representation for the rates implies that those processes are consistent according to $p$ and can thus be extended to $p=\infty$. In this case, we will talk about $\Lambda$-coalescents, taking values in the space $\PS{\infty}$ of partitions of $\N$, which is a compact metric space under a metric $\dist_{\PS{\infty}}$ (see Lemma 2.6 in \cite{Bertoin2006} for the definition of $\dist_{\PS{\infty}}$). We refer the reader to \cite{Bertoin2006} for a thorough exposition of general coalescent processes.
In the present work we leave out the case when $\Lambda$ has an atom at $1$, which corresponds to adding a rate $\Lambda(\{1\})$ at which all the blocks decide to coagulate.

 A famous and important example of $\Lambda$-coalescent processes 
is the family of Beta coalescents \cite{Schweinsberg2003,GnedinIksanovMarynych2014}
in which $\Lambda(d\zeta)=c{\zeta^{1-\beta}(1-\zeta)^{\beta-1}}d\zeta$ for $\beta\in0,2)$.
\subsection{The $\Lambda$-Fleming-Viot processes}  \label{Sec:LFV}
We begin with a few remarks on the space $\sM(\typeS)$ endowed with the topology of weak convergence. 
By Theorem 1.14 in \cite{ZenghuLi2022} the space $\sM(\typeS)$ is locally-compact
whenever $\typeS$ is compact. 
In fact, following the proof of this theorem, the set 
$\sM_r(\typeS)=\{\mu\in \sM(\typeS)\colon \mu(\typeS)\leq r\}$, for $r\geq0$, is compact.\par
Let us write $$\mint{f}{\mu}\coloneqq \int \mu(da)f(a).$$
An important class of functions in $\cf(\sM(\typeS))$ is the algebra of polynomials 
$\polys\left(\sM(\typeS)\right)$ which is the linear span of monomials of the form
\begin{equation}\label{eq:MeasureMonomial}
    F_{\phi,p}(\rho)=\mint{\phi}{\rho^{\otimes p}}, \quad \phi\in \bbf(\typeS^p).
\end{equation}
Here, the space $\typeS^p$ is endowed with the Borel $\sigma$-algebra, which coincides with 
the product $\sigma$-algebra of
$p$ copies of the Borel $\sigma$-albegra on $\typeS$.
By a straightforward extension of Lemma 2.1.2 in \cite{Dawson1993} (extending the arguments 
therein to $\cbf(\sM(\typeS))$), the polynomials $\polys\left(\sM(\typeS)\right)$ are
dense in the topology of uniform convergence on compact sets on $\cf(\sM(\typeS))$, and  convergence determining for
the topology of weak convergence in $\sM(\sM_r(\typeS))$ for
every $r\geq 0$. \par
A function $F\in \cf(\sM(\typeS))$ is said to be differentiable if 
its derivative in the direction of $a\in\typeS$ (more precisely of $\delta_a$) given by
$$
F'(\mu;a) \coloneqq \lim_{\eps\to 0} \frac{F(\mu + \eps \delta_a) - F(\mu)}{\eps},
$$
exists and is continuous as a function of $a\in\typeS$. We denote by $F''(\mu;a,b)$ the
second derivative of $F$, first in the direction of $a$ and then in the direction of $b$; whereas
for higher derivatives 
we write $F^{(\ell)}(\mu;a_1,\cdots,a_\ell)$ for the corresponding sequential derivatives
in the directions of $a_1,\cdots,a_\ell$.
By Lemma 2.1.2 in \cite{Dawson1993}, the polynomials $\polys\left(\sM(\typeS)\right)$ are infinitely differentiable. Their derivatives  can be written in terms of the derivatives of monomials. The first derivative of a monomial is given by
\begin{equation}\label{eq:F'monomial}
    F_{\phi,p}'(\mu;a) = \sum_{i=1}^{p} \mint{\phi}{\mu^{\otimes i-1}\otimes \delta_a \otimes \mu^{p-i-1}}.
\end{equation}
 Multiple derivatives can be computed recursively, for $\vec a =(a_1,\cdots,a_\ell)$, we obtain
\begin{equation}\label{eq:pointDerPolyn}
    F^{(\ell)}_{\phi,p}(\mu;\vec a) = \begin{cases}
     \sum_{\vec m\in \perms(\vec a,\mu,p)}\mint{\phi}{\otimes_{i=1}^p m_i} &\text{ if }\ell\leq p\\
     0 & \text{ otherwise},
    \end{cases}
\end{equation}
where the sum is taken over all the permutations $\vec m\in\perms(\vec a,\mu,p)$, say $\vec m=(m_1,\dots,m_p)$,
of the atomic measures $\delta_{a_1},\dots,\delta_{a_\ell}$ and $(p-\ell)$ copies of $\mu$.
\par
Now, let $\Lambda$ and $\beta_{j,i}^{(\Lambda)}$ be as in section \ref{sec:coalescents}.
The $\Lambda$-Fleming-Viot  process \cite{BertoinLeGall2003} is the process with values in the
space $\sPM(\typeS)$ of probability measures on $\typeS$ 
and generator, applied to functions of the form \eqref{eq:MeasureMonomial}, given by
\begin{align}\label{eq:FVgenerator0}
\FVG F_{\phi,p}(\rho) &= \int_\typeS \rho(da)  \sum_{\ell=2}^p \beta_{p,\ell}^{(\Lambda)} 
\sum_{\vec m\in \perms(\vec a,\mu,p)}\lbr\mint{\phi}{\otimes_{i=1}^p m_i} -\mint{\phi}{\rho^{\otimes p}}\rbr\nonumber\\
&=\int_\typeS \rho(da) \sum_{\ell=2}^p \beta_{p,\ell}^{(\Lambda)}\left( F^{(\ell)}_{\phi,p}(\rho;a)-\binom{p}{\ell} F_{\phi,p}(\rho)\right).
\end{align}
 Here we have made a slight abuse of notation
by writing $F^{(\ell)}_{\phi,p}(\mu;a)$ for the $\ell$-times derivative of $F$, all
in the direction of $a$. The above form of the generator yields the following well-known duality relation between
$\Lambda$-Fleming-Viot processes and $\Lambda$-coalescents. This relation can be extended to a path-wise duality relation via a coupling of both processes that is based on the lookdown construction of \cite{DonnellyKurtz1999}, see section 2 in \cite{les7} for details and also \cite{BirknerBlathEtal2009, GCOSJ} for generalizations.

Let us describe the duality relation. Fix $p\geq 1$ 
and $\phi\in\bbf(\typeS^p)$. 
For any $\pi\in \PS{[p]}$, recall that $\card\pi$ denotes its cardinality, and define
the function $\phi_\pi\colon \typeS^{\card \pi}\to\R$ that results from 
identifying  all the input coordinates  $(a_1,\dots,a_p)$ of $\phi$
according to the blocks of $\pi$. Let $(\pi_1,\dots,\pi_{\card \pi})$ be the enumeration of the blocks of $\pi$
when they are ordered {increasingly} according to their least elements (see e.g. Definition 2.8 in \cite{Bertoin2006}). The function $\phi_\pi$ is given by
\begin{align}\label{eq:defPhiPi}
\phi_\pi(a_1,\cdots,a_{\card \pi}) \coloneqq \phi(a_{\pi(1)},\cdots,a_{\pi(p)}), \text{ where } 
\pi(i) = j \text{ whenever }i\in\pi_j.
\end{align}
Then the duality relation holds for functions of the form 
$$G_\phi(\rho,\LevyMeasure)=\langle \phi_{\pi}, \rho^{\otimes \card\pi}\rangle\equiv\langle \phi_{\pi}, \rho^{\otimes p}\rangle.$$ 
\begin{lemma}[\cite{BertoinLeGall2003,DawsonHochberg1982}]\label{th:FVCoaDuality}
   For fixed $\phi\in\bbf(\typeS^p)$, $p\geq1$, and $\pi=\{\{1\},\cdots,\{p\}\}$; we have
 $$
 \E_{\rho_0}[ \langle \phi_{\pi}, \rho_t^{\otimes p} \rangle ] = \tilde\E_{\pi}[\langle \phi_{\Pi_t}, \rho_0^{\otimes p}\rangle]
 $$
 whenever $\stp{\rho_t}$ is a $\Lambda$-Fleming-Viot process (started at $\rho_0$) under $\P$, 
 and $\stp{\Pi_t}$ is a $(p,\Lambda)$-coalescent process (started at $\pi$) under $\tilde\P$.
\end{lemma}

\par
The case $\Lambda=\sigma^2\delta_0$ corresponds to the standard Fleming-Viot process
without mutation and of parameter $\sigma$ which is dual to Kingman's coalescent and in which \eqref{eq:FVgenerator0} becomes
$$
\FVG F(\rho) = \sigma^2 \int_\typeS \rho(da) (F''(\rho;a,a) - F(\rho)).
$$ 
When $\Lambda(\{0\})=0$, the generator \eqref{eq:FVgenerator0} can be written as in \cite{les7},
\begin{equation*}
\FVG F(\rho)=\int_{{(0,1)}} \frac{\Lambda(\dif\zeta)}{\zeta^2} \int_\typeS \rho(\dif{a}) (F(\rho(1-\zeta)+ \zeta\delta_a)-F(\rho)).
\end{equation*}
 Combining both cases we obtain the following form of the generator
\begin{align}\label{eq:FVgenerator}
\FVG F(\rho) =& \sigma^2 \int_\typeS \rho(da) (F''(\rho;a,a) - F(\rho)) \nonumber\\ 
+&\int_{{(0,1)}} \frac{\Lambda(\dif\zeta)}{\zeta^2} \int_\typeS \rho(\dif{a}) (F(\rho(1-\zeta)+ \zeta\delta_a)-F(\rho)).
\end{align}
\par
When $\zeta^{-2}\Lambda(\dif\zeta)$ is finite, the above form of the generator gives
the following picture for the dynamics of the process.
It has jumps of the form
$\rho_{t-} \to \rho_{t-}(1-\zeta)+ \zeta\delta_a$ at the atoms $(t,\zeta)$ of  a 
Poisson point process on $\Rp\times [0,1]$ with intensity $dt\times \zeta^{-2}\Lambda(\dif\zeta)$.
Here the position $a$ of the new atom of size $\zeta$ is chosen randomly according to $\rho_{t-}$. 
After each jump, the process starts as  an independent copy of a standard Fleming-Viot process of parameter $\Lambda(\{0\})$
started at the new state $\rho_t=\rho_{t-}(1-\zeta)+ \zeta\delta_a$.

\subsection{The Dawson-Watanabe process and its Lamperti transformation}\label{sec:DWP}


Here we introduce the Dawson-Watanabe process without mutation/spatial motion. This will
suffice our applications further ahead; the interested reader can refer to \cite{PerkinsBolthausenVaart2004}
for a more general setting.\par
For a fixed parameter $\sigma\in\R$, the Dawson-Watanabe process  without mutation can be defined
as the unique continuous process $\stp{\mu_t}$ on $\sM(\typeS)$ such that for all $\phi\in\cbf(\typeS)$ the process
$$
\quad M_t(\phi)=\mint{\phi}{\mu_t} - \mint{\phi}{\mu_0}
$$
is a martingale with quadratic variation
$$
\QV{M(\phi)}_t = \frac{\sigma^2}{2}\int_0^t \mint{\phi^2}{\mu_s} ds.
$$
Alternatively, it can be defined as the unique solution to the martingale problem for the operator
$(\DWG, D_\DWG)$ given by
\begin{equation}\label{eq:DWgenerator}
\DWG F(\mu) = \frac{\sigma^2}{2}\int_\typeS \mu(da) F''(\mu;a,a) ,
\end{equation}
and
$D_\DWG\coloneqq \{F(\mu)=f(\mint{\phi_1}{\mu},\cdots,\mint{\phi_p}{\mu})\colon \forall 1\leq i\leq p,~\phi_i\in \cbf(\typeS), f\in \czdf{\infty}(\R)\}\subset\D(\DWG).
$
See Corollary 2.23 in \cite{Fitzsimmons1992}.
Let \begin{equation}\label{eq:defD_DWG'}
	D_\DWG'=\left\{F(\mu)=h(\norm{\mu})\mint{\phi}{\left(\frac{\mu}{\norm{\mu}}\right)^{\otimes p}}: h\in D_h,~ \phi\in\cbf(\typeS^p)\right\},
\end{equation}
where
\begin{equation}\label{eq:defD_h}
	D_h \equiv \{h\in \cKdf{\infty}(\Rp)\colon h(0)=0\text{ and }\restr{h}{(0,\infty)}\in \cKf((0,\infty))\}.
\end{equation}
In section \ref{sec:duality} we prove that $(\DWG, D_\DWG')$ is in the bounded point-wise closure of $(\DWG,D_\DWG)$ (in the sense
that for any $F\in D'_{\DWG}$ the pair $(F,\DWG F)$ can be boundedly point-wise approximated by  $(F_n,\DWG F_n)$ with $F_n\in D_\DWG$), from which we obtain the following result.
\begin{proposition}\label{prop:DWG'subsetDDWG}
	We have $D_\DWG'\subset \D(\DWG).$
\end{proposition}

It is well known that the time-changed frequency process $\stp{\frac{\mu_{c_1(t)}}{\norm{\mu_{c_1(t)}}}}$ is a standard Fleming-Viot process \cite{les7,Perkins1992}.  We end this section with our first application of Theorem \ref{th:smhssbijection}, which complements this result,
and which we will then generalize in Theorem \ref{th:mvpLamperti}. {Recall the process $\stp{\nu_t}$ in the introduction that has generator of the form \eqref{eq:generatormvMp}. The following theorem provides a formal construction of this process in the  case when $\drift=0$ and $\Lambda=0$. The general construction is given in Theorem \ref{th:MVPsmhExist} further below.}
\begin{proposition}\label{prop:DWselfsimilar}
    The process $\stp{\mu_t}$ is a $1$-SS Markov process. The time changed process $\stp{\nu_t}= \stp{\mu_{c_1(t)}}$
    is SMH and its generator $\Gdiff{\sigma}$ has the form
    \begin{equation}\label{eq:mvpSMHdiffgenerator}
    \Gdiff{\sigma}F(\nu)\coloneqq \norm{\nu} \DWG F(\nu)
    \end{equation}
    in the set 
    \begin{align}\label{eq:domainD1}
    {D}_{\Gdiff{\sigma}}\coloneqq\big\{F\in \D(\DWG)\colon 
    & F\in \cdf{2}(\sM(\typeS)) \quad \& \quad  \exists k\geq1,C\geq0\colon
    \nonumber\\
    &\forall a\in\typeS, \nu\in\sM(\typeS);\quad  \abs{F(\nu)} + \abs{F'(\nu;a)} + \abs{F''(\nu;a,a)}
    \leq C(1+\cdots+\norm{\nu}^k)
    \big\}.
    \end{align}
    Furthermore, the process $\stp{\rho_t,\xi_t}=\stp{\frac{\nu_t}{\norm{\nu_t}},\log(\norm{\nu_t})}$ is a MAP 
    with $\stp{\rho_t}$ being a standard Fleming-Viot process of parameter $\sigma$, and $\stp{\xi_t}$ a continuous Lévy process
    with diffusion parameter {$\sigma$} and drift parameter 
    $$\driftF\coloneqq-\sigma.$$
\end{proposition}
\begin{proof}
Taking functions of the form $F(\mu)=f(\mint{1}{\mu})$
with $f\in\czdf{\infty}(\R)$ in the generator \eqref{eq:DWgenerator}, it is easy to see that the total mass process $\stp{\norm{\mu_t}}$ has generator of the form
$\PSSMPG f(x)=\frac{\sigma^2}{2}xf''(x)$ on the set $\{f\in \czdf{2}(\R)\colon xf''(x)\in\cf([0,\infty])\}\subset \D(\PSSMPG)$. The latter
 conforms with the general form of the generator of a continuous positive 1-SS Markov process. 
By Theorem 5.1 in \cite{Lamperti1972} the process $\stp{\norm{\mu_t}}$ is then a uniquely determined diffusion (Feller's diffusion), 
 with absorbing state $0$. Furthermore, by Theorem 4.1 therein
the time-changed process 
$\stp{\log(\norm{\mu_{c_1(t)}})}$ is a continuous Lévy process with diffusion parameter $\sigma$ and drift parameter 
$\driftF=-\sigma.$ 
On the other hand, it is well known  from Theorem 1.1 $i)$ in \cite{les7} (see also \cite{Perkins1992}) that the time-changed frequency process
     $\stp{\rho_t}$ defined by $\rho_t=\frac{\mu_{c_1(t)}}{\norm{\mu_{c_1(t)}}}$
    is a standard Fleming-Viot process of parameter $\sigma$.
\par
We now check that $\stp{\nu_t}$ is characterized by \eqref{eq:mvpSMHdiffgenerator}.
Since $\stp{\norm{\mu_t}}$ is continuous and absorbed at 0 we have
$$
\inf_{s>0}\lbr s\colon \norm{\mu_s} = 0 \rbr \aseq 
\inf_{s>0}\lbr s\colon \int_0^s\frac{1}{\norm{\mu_s}} ds = \infty \rbr,
$$
which implies, through a direct application of Theorem VI.1.3 in \cite{EthierKurtz86},
that $\stp{\nu_t}$ is a
solution to the martingale problem for $\Gdiff{\sigma}$ on the domain
$\{F\in \D(\DWG)\colon \norm{\cdot}\DWG F(\cdot) \in \bbf(\sM(\typeS))\}$. 
Thanks to the fact that
the time-changed process $\stp{\norm{\nu_t}}$ is the exponential of the continuous Lévy process given by
$\xi_t=\driftF t+\sigma B_t$, the domain
can easily be extended to the set $D_{\Gdiff{\sigma}}$ by a mild adaptation of the proof of Theorem VI.1.3 in \cite{EthierKurtz86}.
Indeed, we first note that $N_t=F(\mu_{t}) - \int_0^{t} \DWG F(\mu_{u}) du$ is a martingale whenever $F\in \D(\DWG)$. Furthermore,
 for $F\in D_{\Gdiff{\sigma}}\subset \D(\DWG)$, we have the bound
 
\begin{align*}
&\sup_{0\leq s \leq t}\abs{ N_{c_1(s)}}^2=\sup_{0\leq s\leq t}\abs{F(\mu_{c_1(s)}) - \int_0^{c_1(s)} \DWG F(\mu_{c_1(u)}) du}^2\\
&=\sup_{0\leq s\leq t} \abs{F(\nu_s) - \int_0^{s}\norm{\nu_u}\DWG F(\nu_u) du}^2
\leq 2\norm{F}_\infty^2 + 2\sup_{0\leq s\leq t} \left( \int_0^{s}\norm{\nu_u}\abs{\DWG F(\nu_u)} du \right)^2.
\end{align*}
Since $F\in  D_{\Gdiff{\sigma}}$, the last term in the r.h.s. is bounded by
\begin{align*}
	\sup_{0\leq s\leq t} \left( \int_0^{s}\norm{\nu_u}\abs{\DWG F(\nu_u)} du \right)^2
&\leq t^2\sup_{0\leq s\leq t} C(\norm{\nu_s}+\cdots +\norm{\nu_s}^k)^2
\\&\leq t^2e^{kt}\sup_{0\leq s\leq t} (	1+ \cdots + e^{kB_s})^2
\end{align*}
for some $k\geq1$. By Doob's $\Lp{2k}$ inequality applied to the submartingale
$\stp{e^{2kB_t}}$ we obtain
\begin{align}\label{eq:DiffSMHnormBound}
\E\left[\sup_{0\leq s\leq t} \abs{ \int_0^{s}\norm{\nu_u}\DWG F(\nu_u) du }\right]
&\leq\E\left[\sup_{0\leq s\leq t} \left( \int_0^{s}\norm{\nu_u}\abs{\DWG F(\nu_u)} du \right)^2\right]
\nonumber\\
&\leq t^2 e^{kt}\E[(\sup_{0\leq s\leq t} 1 + \cdots + e^{kB_s})^2]<\infty
\end{align}
so that $\E\left[\sup_{0\leq s \leq t}\abs{ N_{c_1(s)}}^2\right]<\infty.$
Using H\"older's inequality we also obtain
\begin{align*}
  \lim_{T\to\infty}  \E[\abs{N_{T}}; c_1(t)>T]
  &\leq\lim_{T\to\infty} \E[\sup_{0\leq s \leq t}\abs{ N_{c_1(s)}}^{2}]^{1/2} \P(c_1(t)>T)^{1/2}
  = 0
\end{align*}
where we have used $c_1(t)\asle\infty$. The optional sampling theorem (e.g. Theorem II.2.13 in \cite{EthierKurtz86})
then implies that $\stp{N_{c_1(t)}}$ is a $\F_{c_1(t)}$-martingale. Moreover
$$
\der{t}\E_\nu[F(\nu_t)] = \der{t}\E_\nu\left[\int_0^t \norm{\nu_u}\DWG F(\nu_u) du\right] = \norm{\nu}\DWG F(\nu)
$$
where we have used dominated convergence using the bound in \eqref{eq:DiffSMHnormBound}.

The MAP property for $\stp{\rho_t,\xi_t}$ will follow from Theorem \ref{th:smhssbijection} once we prove that $\stp{\mu_t}$ is 1-SS. The latter follows from 
Proposition \ref{prop:smhssGenerators}. As stated before, the solutions to the martingale problem for
$(\DWG,D_\DWG)$ are unique.
Also note that
if $H\in\polys(\sM(\typeS))$ then $\SSS_b H\in \polys(\sM(\typeS))$ which implies $\SSS_b D_\DWG\subset D_\DWG$. It thus remains to verify $\ref{itTh:ssGenerator}$ in Proposition \ref{prop:smhssGenerators}. Note that
$$
\DWG\SSS_bF(\mu) =  \frac{\sigma^2}{2}\int_\typeS \mu(da) (\SSS_bF)''(\mu;a,a)
=  \frac{\sigma^2}{2}\int_\typeS \mu(da) b^2F''(b\mu;a,a)
$$
so that
\begin{equation}\label{eq:DWgenIsSS}
\SSS_{b^{-1}}\DWG\SSS_bF(\mu) = b\frac{\sigma^2}{2}\int_\typeS \mu(da)F''(\mu;a,a)=b\DWG F(\mu).
\end{equation}
\end{proof}

\section{Main results for self-similar populations}\label{sec:genealogySSMVps}
\subsection{Construction and duality theorems}\label{sec:ssmvpsMainResults}
The construction of the Feller process $\stp{\nu_t}$ having generator of the form \eqref{eq:generatormvMp}
is split
into several intermediary results. We first provide a Poissonian construction 
when the measure $\zeta^{-2}\Lambda(d\zeta)$ is finite, and then we extend the construction to any finite measure $\Lambda$ through a weak limit, the convergence of the generators, and the identification of the limit. This construction is made explicit in section \ref{sec:nucons}.
\begin{theorem}\label{th:MVPsmhExist}
	There exists a Feller SMH process with generator $(\Gall,\D(\Gall))$ of the form \eqref{eq:generatormvMp}
	on the set $D_\Gall\subset\D(\Gall)$ given by
	\begin{align}\label{eq:coreSMH}
		D_\Gall 
		\coloneqq\big\{F\in \D(\DWG)\colon  
		& F\in \cdf{2}(\sM(\typeS)),\quad \text{and}\quad \exists C\geq0\colon
		\nonumber\\
		& \forall a\in\typeS, \nu\in\sM(\typeS);\quad  \abs{F(\nu)} + \norm{\nu}\abs{F'(\nu;a)} + \norm{\nu}^2\abs{F''(\nu;a,a)}
		\leq C
		\big\}.
	\end{align}
	It is also the unique solution to the martingale problem for $(\Gall, D'_{\Gall})$ where $D'_{\Gall}$ is given by
	\begin{align}\label{eq:GDualityDomain}
		D_{\Gall}' \coloneqq \big\{F(\nu) = G_{p,\phi,h}(\nu,\tilde\pi,z) \colon\quad
		&h\in D_h,\nonumber\\
		&p\geq 1, \phi\in\cbf(\typeS^p), \tilde\pi \in\PS{\infty}, z\in \Rp
		\big\} 
	\end{align}
	and satisfies $D'_{\Gall}\subset D_{\Gall}.$
	\par
	Moreover, the process $(\rho_t,\xi_t)_{t\geq 0}$ where $\rho_t=\nu_t/\norm{\nu_t}$
	and $\xi_t=\log(\norm{\nu_t})$ is a MAP. Furthermore, $\stp{\rho_t}$ is a $\Lambda$-Fleming-Viot process,
	whereas $\stp{\xi_t}$ is a Lévy process with characteristic triplet $(\driftF+\drift, {\sigma}, \LevyMeasure)$, {where we recall $\driftF=-\sigma$ and $\LevyMeasure(d\zeta)$ is the pushforward of the measure $\zeta^{-2}\Lambda(d\zeta)$ under the transformation $\zeta\to -\log(1-\zeta)$ on $(0,1)$}.
\end{theorem}
For the identification of the limit a new duality relation will be needed.
The dual process for $\stp{\nu_t}$ is a two coordinate process $(\Pi_t,Z_t)_{t\geq0}$ taking values  in $\PS{\infty}\times \Rp$. Its
dynamics are characterized by the parameters $(\drift,\sigma,\Lambda)$ appearing in \eqref{eq:generatormvMp}.  
The first coordinate $\stp{\Pi_t}$ is a $(\Lambda+\sigma^2\delta_0)$-coalescent process on $\PS{\infty}$. The second coordinate is such that  $\stp{\log Z_t}$ is a Lévy process that can be characterized by the triplet 
$(\drift,\sigma,\Lambda)$. More precisely, its Lévy
 exponent is of the form
     \begin{equation}\label{eq:LevyCharacteristic}
\Psi(\theta)=i(\driftF+\drift)\theta + \frac{\sigma^2}{2}\theta^2 +   
\int_{0}^{1}\frac{\Lambda(d\zeta)}{\zeta^2} \lbr 1 - e^{-i\theta\log(1-\zeta)} - i\theta\abs{\log(1-\zeta)}\Ind{\zeta\le 1/2} \rbr.
\end{equation} 
In the standard notation, the Lévy measure is $\LevyMeasure(d\zeta)$. Interestingly, the two processes $\stp{Z_t}$ and $\stp{\norm{\nu_t}}$ are equal in law. \par
The two processes $\stp{\Pi_t}$ and $\stp{\log Z_t}$ are coupled through a common Poisson point process of intensity $dt\otimes\zeta^{-2}\Lambda(d\zeta)$. When $\zeta^{-2}\Lambda(d\zeta)$ is finite this point process has finitely many points on any bounded time interval, which can be enumerated increasingly according to the first (time) coordinate. These points drive the sequential coagulations of $\stp{\Pi_t}$ and the jumps of $\stp{\log Z_t}$. 
The construction is carefully described in section \ref{sec:dualConstruction}.


The generator of the process $\stp{\Pi_t,Z_t}$ has the form $\Gdall$ on the set $D_{\Gdall}$ which we define in  \eqref{eq:dualGenerator} and \eqref{eq:dualGeneratorDomain} further below.  In order to avoid the repetition of long mathematical expressions, we define the operator $(\Gdall,D_{\Gdall})$ in terms
of a  ``bridge'' operator $(\GHall,D_{\GHall})$ acting on functions $G\in \bbf\left(\sM(\typeS),\PS{\infty},\Rp\right)$. The operator $(\GHall,D_{\GHall})$ captures a duality relation, at the level
of operators, between $(\Gall,D_{\Gall})$ and $(\Gdall,D_{\Gdall})$ (see Lemma \ref{le:operatorDuality} below). 
The operator $\Gdall$ will then act in
the same way as $\GHall$ but on functions
of the form $G(\nu,\cdot,\cdot)$ with $\nu\in\sM(\typeS)$ held fixed. 
Furthermore, the set ${D_{\GHall}}\subset \bbf\left(\sM(\typeS)\times \PS{\infty}\times \Rp\right)$ on which
$\GHall$ will be defined, will also correspond to the set of functions for which the processes $\stp{\nu_t}$ and $\stp{\Pi_t, Z_t}$ will be in duality (Theorem \ref{th:duality}).
\par  Let us then introduce  $(\GHall,D_{\GHall})$. To ease notation in the following, $\tilde\pi$ will refer to an arbitrary partition in $\PS{\infty}$; whereas $\pi$ will refer to
a corresponding restricted partition 
\begin{equation}\label{eq:notationpi}
	\pi=\restr{\tilde \pi}{p}
\end{equation} for some arbitrary $p\geq1$. Further, for a partition $\pi\in\PS{[p]}$
and a subset $J\subset [p]$, let $\pi^{(J)}$ be the partition formed by coagulating
the blocks $(\pi_i)_{i\in J}$ into a single new block. Also
set $\pi^{(\emptyset)}=\pi=\pi^{(\{i\})}$ for all $i\in[p]$. Let us also write $\{J\subset [p]\colon \card J=0\}\equiv\{\emptyset\}$. \par
Consider functions $G_{p,\phi,h}\in \bbf\left(\sM(\typeS),\PS{\infty},\Rp\right)$ of the  form:
\begin{align*}
	G_{p,\phi,h}(\nu,\tilde\pi,z) &\coloneqq h(\norm{\nu} z)H^{(\phi)}_\pi(\nu),\nonumber\\
	H^{(\phi)}_\pi(\nu)&\equiv H_\pi(\nu)\coloneqq\mint{\phi_\pi}{\left(\frac{\nu}{\norm{\nu}}\right)^{\otimes p}}
,\nonumber
\end{align*}
where  $p\geq1,\phi\in\cbf(\typeS^p)$ and $h\in D_h$ are arbitrary (recall that $D_h$ is defined in \eqref{eq:defD_h}).
The operator 
$(\GHall,D_{\GHall})$ is defined for this type of functions by
\begin{align}\label{eq:jointGenerator}
	\GHall G_{p,\phi,h}(\nu,\tilde\pi,z) &\coloneqq \drift \norm{\nu}zh'(\norm{\nu}z)H_\pi(\nu) + \frac{\sigma^2}{2}\norm{\nu}^2z^2h''(\norm{\nu} z)H_\pi(\nu)
	\nonumber\\
	&+\sigma^2h(\norm{\nu}z) \sum_{\substack{J\subset [p]\\ \card J=2}} \lbr   H_{\pi^{(J)}}(\nu)
	-   H_{\pi}(\nu)\rbr\nonumber\\
	&+\int_{(0,1)}\frac{\Lambda(d\zeta)}{\zeta^2}\Bigg\{\sum_{\ell = 0}^{p}  \sum_{\substack{J\subset [p]\\ \card J=\ell}} (1-\zeta)^{p-\ell}\zeta^{\ell} \left(    h\left(\frac{\norm{\nu}z}{1-\zeta}\right) H_{\pi^{(J)}}(\nu)
	- h(\norm{\nu}z)  H_{\pi}(\nu)\right)\nonumber\\
	&\quad\quad\quad\quad\quad\quad- \norm{\nu}zh'(\norm{\nu}z)H_\pi(\nu)\abs{\log(1-\zeta)}\Ind{\zeta\le 1/2}\Bigg\};
\end{align}
on the set 
\begin{align}\label{eq:domainJointGenerator}
	D_{\GHall}\coloneqq  \big\{G_{p,\phi,h}(\nu,\tilde\pi,z)\colon \quad
	& h\in D_h, 
	p\geq 1,\phi\in\cbf(\typeS^p)
	\big\}.
\end{align}
In Lemma \ref{le:dualGeneratorBound} we prove that this is always well defined, and in fact uniformly bounded 
on $(\nu,\pi,z)$,
whenever $G_{p,\phi,h}\in D_{\GHall}$. \par
Having introduced $\GHall$, the generator $\Gdall$ of
the process $\stp{\Pi_t,Z_t}$ is defined on 
\begin{align}\label{eq:dualGeneratorDomain}
	D_{\Gdall} \coloneqq \big\{G(\tilde\pi,z) = G_{p,\phi,h}(\nu,\tilde\pi,z) \colon\quad
	&h\in D_h,
	p\geq 1, \phi\in\cbf(\typeS^p), \nu\in\sM(\typeS)
	\big\}. 
\end{align}
by
\begin{align}\label{eq:dualGenerator}
	\Gdall G(\tilde\pi,z) &\coloneqq \GHall G_{p,\phi,h}(\nu,\tilde\pi,z).
\end{align}

\begin{theorem}\label{th:FellerDual}
	There exists a Feller process $\stp{\Pi_t, Z_t}$ with generator of the form $\Gdall$ on $D_{\Gdall}$ which is the unique solution to the martingale
	problem for $(\Gdall,D_{\Gdall})$. Moreover, the process $\stp{\log(Z_t)}$ is a Lévy process as that in Theorem \ref{th:MVPsmhExist}, whereas $\stp{\Pi_t}$ is a $(\Lambda+\sigma^2\delta_0)$-coalescent.
\end{theorem}
\par
As mentioned before, a key element in the proofs of Theorems \ref{th:MVPsmhExist} and \ref{th:FellerDual} is the following duality relation, which is also of interest in itself. 
\begin{theorem}[Duality]\label{th:duality}
	For $G_{p,\phi,h}\in D_{\GHall}$ we have
	\begin{equation}\label{eq:duality}
		\E_{\nu}\left[G_{p,\phi,h}\left(\nu_t,\tilde \pi,z\right)\right]
		=\tilde\E_{(\tilde\pi,z)}\left[G_{p,\phi,h}\left(\nu,\Pi_t,Z_t\right)\right]
	\end{equation}
	whenever $\stp{\nu_t}$ is a solution to the martingale problem for $(\Gall,D_\Gall')$ under $\P$, and
	$\stp{\Pi_t,Z_t}$ is a solution to the martingale problem for $(\Gdall,D_\Gdall)$ under $\tilde \P$.
	 Furthermore, the solutions to the martingale problems for $(\Gall,D_\Gall')$ and $(\Gdall,D_\Gdall)$ are unique.
\end{theorem}

\begin{remark}
	Theorem \ref{th:duality} is an extension of the classical duality relation between $\Lambda$-Fleming-Viot processes and $\Lambda$-coalescents in 
	\cite{BertoinLeGall2003}, to the case of populations with varying size. Indeed, one recovers the duality in \cite{BertoinLeGall2003}
	by setting $h\equiv1$ in \eqref{eq:duality}. 
\end{remark}

\subsection{Lamperti Transformation}\label{sec:MVLampertiTransformation}
Having constructed the SMH process $\stp{\nu_t}$, our main result of this section is the following 
theorem which, on the one hand, generalizes Proposition 2 in \cite{KyprianouPardo2008} to $\beta$-stable 
measure-valued processes $(\beta>1)$
and, on the other, re-frames Theorem 1 in \cite{les7} as a Lamperti transformation 
in the case $\beta\in(1,2]$. The latter comprise the intersection between branching and self-similar processes {of index $\alpha=\beta-1$} (see Remark \ref{rem:alfaStableCase} below). Finally, the following theorem also characterizes the process of frequency of types of a population whose total size evolves as a positive self-similar Markov process  with non-negative jumps, through the well-known
duality relationship between $\Lambda$-Fleming-Viot processes and $\Lambda$-coalescents \cite{BertoinLeGall2003} (see Lemma \ref{th:FVCoaDuality} and Remark \ref{rem:PSSMPgenerator}). This is a first step towards characterizing their genealogies via path-wise duality relations based on lookdown constructions \cite{DonnellyKurtz1999}  (see e.g. \cite{les7} for such a characterization for $\beta$-stable branching processes).\par

\begin{theorem}[Lamperti Transformation]\label{th:mvpLamperti}
	Let $\stp{\nu_t}$ be the Feller SMH process of Theorem \ref{th:MVPsmhExist}.
	\begin{enumerate}   
		\item \label{itTh:mvpsmhtoss} Let $\alpha\geq 0$ and recall the random time change $\gamma_\alpha(t)$ 
		of Theorem \ref{th:smhssbijection}. The time-changed process 
		$\stp{\mu_t}= \stp{\nu_{\gamma_\alpha(t)}}$ is the unique solution to the following integral equation:
		for
		$$
		S_\mu \coloneqq   \sup\left\{\int_0^t \norm{\mu_u}^{-\alpha} du \colon \int_0^t \norm{\mu_u}^{-\alpha} du <\infty \right\},
		$$
		we have
		{
			\begin{equation}\label{eq:SSSMH-equation}
				\stp{\mu_t} = \stp{\nu_{\int_0^t \norm{\mu_s}^{-\alpha}ds\wedge S_\mu}}.
			\end{equation}
		}
		Furthermore, it is a $\alpha$-SS standard Markov process with generator of the form  $\Fall_{\alpha}$ in  \eqref{eq:generatormvssMp}
		on the set 
		$$
		\{F\in D_\Gall\colon \norm{\cdot}^{-\alpha}\Gall{F}(\cdot) \in\bbf(\mathcal M(\typeS))
		\}\subset \D(\Fall_\alpha).
		$$
	\item \label{itTh:mvpsstomap} Recall the time change $c_\alpha(t)$ of Theorem \ref{th:smhssbijection}.
	If $\stp{\mu_t}$ is the (unique) solution to \eqref{eq:SSSMH-equation}, then 
	$
	\stp{\nu_t}=\stp{\mu_{c_\alpha(t)}}.
	$
	Furthermore, $\stp{\nu_t}$ is the unique
	solution to 
	$$
	\stp{\nu_t}=\stp{\mu_{\int_0^t \norm{\nu_s}^\alpha ds\wedge S_\nu}},
	$$
	where 
	$$
	S_\nu\coloneqq \sup\left\{\int_0^t \norm{\nu_u}^\alpha du \colon \int_0^t \norm{\nu_u}^\alpha du <\infty \right\}.
	$$
	
	
\end{enumerate}
\end{theorem}

\begin{proof}
Item $\ref{itTh:mvpsmhtoss}$ follows from first applying Theorem \ref{th:smhssbijection}
to obtain that $\stp{\mu_t}$ is a standard $\alpha$-SS Markov process, and also that it is the unique solution 
to \eqref{eq:SSSMH-equation}.
Now, observe that $\stp{\norm{\nu_t}^{-\alpha}}$ is a.s. bounded on bounded
time intervals since $\stp{\xi_t}$ is a.s. bounded away from $-\infty$ on bounded
time intervals. The latter also implies that
$\inf\{t\colon \norm{\nu_t}^{-\alpha}=0\}\aseq\infty=\gamma_\alpha(\infty)$.
Then 
by Theorem VI.1.3 in \cite{EthierKurtz86}, the time-changed process
$\stp{\mu_t}$  is a solution to the martingale problem for \eqref{eq:generatormvssMp} 
on the set 
$$
\{F\in D_\Gall\colon \norm{\cdot}^{-\alpha}\Gall{F}(\cdot) \in\bbf(\mathcal M(\typeS))
\}\subset \D(\Fall_\alpha).
$$
Item $\ref{itTh:mvpsstomap}$ is a direct consequence of Theorem \ref{th:smhssbijection}$ \ref{thIt:mhtossEq}$ and Corollary \ref{Th:smhtossTimeChangeBijection}. 
\end{proof}

\begin{remark}\label{rem:PSSMPgenerator}
As in section \ref{sec:DWP}, when plugging functions of the form $F(\mu)=f(\mint{1}{\mu})$, where $f(x),xf'(x),x^2f''(x)\in\czf(R)$, into  \eqref{eq:generatormvssMp}, we obtain
that the total mass process  $\stp{\norm{\mu_t}}$ is a positive self-similar Markov process with generator of the form
\begin{align}\label{eq:PSSMPGenerator}
	\PSSMPG f(x) =& 
	{(\driftF+\drift)} x^{1-\alpha} f'(x)+ \frac{\sigma^2}{2}x^{2-\alpha}f''(x)  \nonumber\\
	+&x^{-\alpha}\int_1^\infty [f(x\zeta)-f(x) - f'(x)\log(\zeta)\Ind{\zeta<1/2}] \Theta(d\zeta).
\end{align}
Here we recall $\driftF=-\sigma$. Also $\Theta(d\zeta)$ is the pushforward of the measure $\zeta^{-2}{\Lambda(d\zeta)}$
under the transformation $\zeta\to {1}/({1-\zeta})$ 
(c.f. Theorem 6.1 in \cite{Lamperti1972}).\par 
\end{remark}
\begin{remark}\label{rem:alfaStableCase}
Taking $\sigma=0$, and $\Lambda(d\zeta)=c{\zeta^{1-\beta}(1-\zeta)^{\beta-1}}d\zeta$ with $c>0$ and $\beta\in(1,2)$;
and also
\begin{align*}\drift &= -\int_0^1\frac{\Lambda(d\zeta)}{\zeta^2}\lbr  \frac{\zeta}{1-\zeta} - (\abs{\log(1-\zeta)}\Ind{\zeta<1/2}) \rbr,
\end{align*}
in \eqref{eq:generatormvssMp}, we obtain 
\begin{align*}
{\Fall}(\nu)
=\frac{1}{\norm{\nu}^{\beta-1}}\int_{\typeS} \frac{\nu(da)}{\norm{\nu}} \int_0^1\frac{\Lambda(d\zeta)}{\zeta^2} \bigg\{ &F\left(\nu + \norm{\nu}\frac{\zeta}{1-\zeta} \delta_a\right)
- F(\nu)\\
&- \norm{\nu} \frac{\zeta}{1-\zeta} F'(\nu; a)\bigg\}.
\end{align*}
The latter, after the change of variable $h= \norm{\nu} \frac{\zeta}{1-\zeta}$, becomes
$$
{\Fall}(\nu) = c\int_\typeS\nu(da)\int_0^\infty h^{-1-\beta}dh \bigg\{ F\left(\nu + h \delta_a\right)
- F(\nu) -  h F'(\nu; a)\bigg\},
$$
the generator of a $\beta$-stable measure-valued branching process with $\beta>1$. The case
$\beta=2$ is covered by picking $\Lambda(d\zeta)=c\delta_0(d\zeta)$ instead. Thus, Theorem \ref{th:mvpLamperti} \ref{itTh:mvpsmhtoss} generalizes Proposition 2 \cite{KyprianouPardo2008} to the measure-valued setting, while Theorem \ref{th:mvpLamperti} \ref{itTh:mvpsstomap} recovers Theorem 1 i) and ii) (for $\beta>1$) in \cite{les7}.\par
\end{remark}

\section{Technical lemmas}\label{sec:technicalAndGenerators}
\subsection{Maximal inequality for exponential of Lévy processes}
The following lemma will be used to show that the process $\stp{\nu_t}$ is tight. 
\begin{lemma}\label{le:LevyExponentialMoments}
    Let $\stp{\xi_t}$ be a Lévy process with characteristic exponent of the form \eqref{eq:LevyCharacteristic}.
    Let $q\ge 1$ 
    and assume  that $e^{\xi_0}$
    is $\Lp{q}$-bounded. Then, 
    \begin{align*}
    \E\left[\sup_{0\leq t\leq T}e^{q \xi_t}\right] & 
    < C_q\exp\left\{ TC'_q\left( \abs{\drift+\driftF} +\sigma^2+ \norm{\Lambda}+  \int_{1/2}^1 \frac{\Lambda(d\zeta)}{\zeta^2} \left\{ e^{-q\log(1-\zeta)}-1\right\}  \right) \right\},
    \end{align*}
    for some constants $C_q>0$ and $C'_q>0$.
\end{lemma}
\begin{proof}
    Let
$\stp{M_t}$ be the  
 martingale
 component of $\stp{\xi_t}$, given by 
$$
M_t=\sigma B_t + \int_0^t\int_{0<\zeta\le 1/2} \abs{\log(1-\zeta)} \tilde\PPP(ds,d\zeta),
$$ 
where $\stp{B_t}$ is a Brownian motion and
  $\tilde\PPP$ is a compensated Poisson random measure on $\Rp\times(0,1/2)$ of intensity $ds\times\zeta^{-2}\Lambda(d\zeta)$. 
  Let also $\stp{J_t}$ be the pure-jump component of $\stp{\xi_t}$ given by
  $$
  J_t=\int_0^t\int_{1/2\leq \zeta<1} \abs{\log(1-\zeta)} \PPP(ds,d\zeta)
  $$
  where $\PPP$ is a Poisson random measure on $\Rp\times[1/2,1)$ of intensity 
  $ds\times\zeta^{-2}\Lambda(d\zeta)$.
  Then
  \begin{align}\label{eq:normExpectationBound0}
  \E\left[\sup_{0\leq t\leq T}e^{q\xi_t}\right]
  &\leq e^{\abs{\drift+\driftF}T}\E[e^{q\xi_0}]\E[e^{qJ_T}]\E\left[\sup_{0\leq t\leq T}e^{qM_t}\right].
  \end{align}
  The expectation $\E[e^{q\xi_0}]$ is bounded by hypothesis. 
  Also, by Campbell's formula  we have
  $$
    \log(\E[e^{qJ_T}]) = T\int_{1/2}^{1}\frac{\Lambda(d\zeta)}{\zeta^2} \lbr e^{-q\log(1-\zeta)}-1 \rbr.
  $$
  We now bound the last expectation in the r.h.s. of \eqref{eq:normExpectationBound0}
By Doob's $\Lp{q}$-maximal 
inequality applied to the submartingale $\stp{e^{qM_t}}$  we have for some constant $C_q$, and plugging Campbell's formula in the second inequality below,
\begin{align}\label{eq:normExpectationBound}
&\E\left[\sup_{0\leq t\leq T} e^{qM_t}\right]\leq C_q\E\left[e^{qM_T}\right]\\
&\leq  C_q e^{Tq^2	\frac{\sigma^2}2} \exp\lbr T\int_{0}^{1/2}\frac{\Lambda(d\zeta)}{\zeta^2} \lbr e^{-q\log(1-\zeta)}-1+q\log(1-\zeta) \rbr \rbr.\nonumber
\end{align}
Since the integrand above is of order $\bO(\zeta^2)$ as $\zeta\to0$, we have
$$
\int_{0}^{1/2}\frac{\Lambda(d\zeta)}{\zeta^2} \lbr e^{-q\log(1-\zeta)}-1+q\log(1-\zeta) \rbr <C\norm{\Lambda}<\infty,
$$
for some $C>0$. Putting all the bounds together we obtain the result.
\end{proof}
\subsection{Regularity results for generators}
The following lemma is used in section \ref{sec:nucons} to prove that the weak limit of solutions to the martingale problems for $(\Gall_n, D_\Gall)$ defined as in \eqref{eq:generatormvMp}
with corresponding jump measure $\Lambda_n$, is a solution to the martingale problem for $(\Gall, D_\Gall)$ with jump measure $\Lambda=\lim_{n\to\infty}\Lambda_n$.
\begin{lemma}\label{le:boundOnGLambda}
    For $F\in D_\Gall$ and finite measures $\Lambda,\Lambda'$ on $(0,1)$ we have
    \begin{equation}\label{eq:GjumpBound}
    \norm{\Gjump{\Lambda} F - \Gjump{\Lambda'} F}_\infty \leq C \norm{\Lambda-\Lambda'}_{TV}
    \end{equation}
    for some $C>0$, where $\norm{\Lambda-\Lambda'}_{TV}$ is the total variation between $\Lambda$ and $\Lambda'$.
    In particular, for the corresponding opeartors $\Gall=\Gdrift{\drift} + \Gdiff{\sigma} + \Gjump{\Lambda}$ and 
    $\Gall'=\Gdrift{\drift} + \Gdiff{\sigma} + \Gjump{\Lambda'}$ we have
    $$
    \norm{\Gall F - \Gall' F}_\infty \leq C \norm{\Lambda-\Lambda'}_{TV}.
    $$
    \par
    Also, for any finite measure $\Lambda$ on $(0,1)$, and for the corresponding operator $\Gall$, we have for any $ F\in D_\Gall$,
    \begin{equation}\label{eq:GallBound}
     \norm{\Gall F}_\infty \leq C\norm{\Lambda}_{TV}<\infty. 
    \end{equation}
\end{lemma}
\begin{proof}
	In what follows $C$ will be a constant
	that may vary from line to line but that does not depend on $\nu$.
    On the one hand, whenever $F\in D_\Gall$ the diffusion and drift terms of $\Gall F$ are 
bounded
\begin{align}\label{eq:GdiffDriftBound}
\abs{\Gdiff{\sigma}F(\nu)}+\abs{\Gdrift{\drift}F(\nu)}&\leq \norm{\nu} \int_\typeS \nu(da)\frac{\sigma^2}{2} \abs{F''(\nu;a,a)}+\int_\typeS \nu(da)\kappa \abs{F'(\nu;a)}\nonumber\leq C.
\end{align}                   
Thus \eqref{eq:GallBound} easily follows from $\eqref{eq:GjumpBound}$.
Let us then prove \eqref{eq:GjumpBound}. 
Fix a measure $\nu\in\sM(\typeS)$ and regard 
$F(\nu + \norm{\nu}\frac{\zeta}{1-\zeta}\delta_a)$ as a function of $\zeta$. Then writing 
$\psi(\zeta)=\frac{\zeta}{1-\zeta}$ so that $\psi'(\zeta)=\frac{1}{(1-\zeta)^2}$ and $\psi''(\zeta)=\frac{2}{(1-\zeta)^3}$,
we have
\begin{align*}
&\der{\zeta}F(\nu+\norm{\nu}\psi(\zeta)\delta_a)\\
&= \lim_{h\to 0} \frac{F(\nu+\norm{\nu}\psi(\zeta)\delta_a + \norm{\nu}(\psi(\zeta+h)-\psi(\zeta))\delta_a) - 
                   F(\nu+\norm{\nu}\psi(\zeta)\delta_a)}
                       {\norm{\nu}(\psi(\zeta+h)-\psi(\zeta))} \\
&\times\lim_{h\to 0}\norm{\nu}\frac{\psi(\zeta+h)-\psi(\zeta)}{h}\\
&= \norm{\nu}F'(\nu+\norm{\nu}\psi(\zeta)\delta_a;a)\frac{1}{(1-\zeta)^2}.
\end{align*}
Similarly,
\begin{align*}
\der{\zeta^2}F(\nu+\norm{\nu}\psi(\zeta)\delta_a)
&= \norm{\nu}^2F''(\nu+\psi(\zeta)\delta_a;a,a)\left(\frac{1}{(1-\zeta)^2}\right)^2  \\
&+\quad\norm{\nu}F'(\nu+\psi(\zeta)\delta_a;a)\frac{2}{(1-\zeta)^3}.
\end{align*}
Using the assumption $F\in D_\Gall$, the latter is uniformly bounded on $\nu\in\sM(\typeS),a\in\typeS$, by
$$
\der{\zeta^2}F(\nu+\norm{\nu}\psi(\zeta)\delta_a) \leq C\left(\frac{1}{(1-\zeta)^4}+\frac{1}{(1-\zeta)^3}\right).
$$
Taylor's expansion for $\zeta$ near $0$ then gives
$$
\abs{F\left(\nu + \norm{\nu}\frac{\zeta}{1-\zeta}\delta_a\right) - F(\nu) - 
\norm{\nu}\zeta F'\left(\nu;a\right)}
\leq C\zeta^2\left(\frac{1}{(1-\zeta)^4}+\frac{1}{(1-\zeta)^3}\right).
$$
Note also that
\begin{equation}\label{eq:zetalogIneq}
\forall 0<\zeta<1,\quad \abs{\zeta-\abs{\log(1-\zeta)}} = \sum_{k=2}^\infty \frac{\zeta^k}{k}
\leq \frac{\zeta^2}{1-\zeta}.
\end{equation}
Combining the above two bounds on the range $\zeta<1/2$, and  using that
$\norm{F}_\infty<\infty$ (since $F\in D_\Gall$) on the range $\zeta>1/2$, we obtain 
$$
\sup_{\zeta\in(0,1)}\frac{1}{\zeta^2}\abs{\mint{F\left(\nu + \norm{\nu}\frac{\zeta}{1-\zeta}\delta_a\right) - F(\nu) - 
\norm{\nu}\abs{\log(1-\zeta)}\Ind{\zeta<1/2} F'(\nu;a)}{\frac{\nu(da)}{\norm{\nu}}}}<C.
$$
Plugging this bound into both $\Gjump{{\Lambda'}}F(\nu)$ and $\Gjump{{\Lambda}}F(\nu)$, we obtain \eqref{eq:GjumpBound}.
\end{proof}

The following Lemma is used to prove the convergence of the generators $\Gdall_n$ of the dual processes $\stp{\Pi_t^{(n)}, Z_t^{(n)}}$ defined for finite
$\zeta^{-2}\Lambda_n(d\zeta)$, to the operator $\Gdall$ defined with the limit $\Lambda=\lim_{n\to\infty}\Lambda_n$. We note that, since in fact we work with the ``joint'' operator $\GHall$, Lemma \ref{le:dualGeneratorBound} could also be used to derive Lemma \ref{le:boundOnGLambda}, albeit for a smaller family of functions. 
\begin{lemma}\label{le:dualGeneratorBound}
	Whenever
	$G_{p,\phi,h}\in D_{\GHall}$ we have $\norm{\GHall G_{p,\phi,h}}_{\infty}<\infty$.  
	Furthermore, for any pair of finite measures $\Lambda,\Lambda'$ on $(0,1)$ and corresponding operators $\GHall,\GHall'$ we have, for
	fixed $G_{p,\phi,h}\in D_{\GHall}\equiv D_{\GHall'}$ and some $C>0$, that	
	$$
	\norm{\GHall G_{p,\phi,h}- \GHall' G_{p,\phi,h}}_\infty\leq C\norm{\Lambda- \Lambda'}_{TV}.
	$$
\end{lemma}
\begin{proof}
	Let us bound the third term in \eqref{eq:jointGenerator}. In what follows $C$ will be a constant
	that may vary from line to line, that only depends on $p,\norm{\phi}_\infty, \norm{h(x)}_{\infty}, \norm{xh'(x)}_\infty$ and $\norm{x^2h''(x)}_{\infty}$, but not on $\nu,\pi=\restr{\tilde\pi}{p}$ nor $z$. 
	Let us first bound the integrand appearing in  the third term of \eqref{eq:jointGenerator} within the interval  $\zeta \in (0,1/2)$.
	Starting the sum that appears inside the integral from the index $\ell=2$, and using  that $h\in\cbf(\R)$ and $H_\pi(\nu)\in\bbf(\sM(\typeS))$, we obtain:
	\begin{align*}
&\sum_{\ell = 2}^{p}  \sum_{\substack{J\subset [p]\\ \card J=\ell}} (1-\zeta)^{p-\ell}\zeta^{\ell} \lbr    h\left(\frac{\norm{\nu}z}{1-\zeta}\right) H_{\pi^{(J)}}(\nu)
		- h(\norm{\nu}z)  H_{\pi}(\nu)\rbr\\
		&\leq \sum_{\ell = 2}^{p}  \sum_{\substack{J\subset [p]\\ \card J=\ell}} (1-\zeta)^{p-\ell}\zeta^{\ell}   C<  C\zeta^2.
	\end{align*}
	\par For the remaining two indices $\ell\in\{0,1\}$ we have $H_{\pi^{(J)}}(\nu)=H_{\pi}(\nu)$
	whenever $J\subset[p]$ and $\card J=\ell$. Thus, 
	\begin{align}\label{eq:boundGdual}
		&\sum_{\ell\in\{0,1\}}  \sum_{\substack{J\subset [p]\\ \card J=\ell}} (1-\zeta)^{p-\ell}\zeta^{\ell} \left(    h\left(\frac{\norm{\nu}z}{1-\zeta}\right) H_{\pi^{(J)}}(\nu)
		- h(\norm{\nu}z)  H_{\pi}(\nu)\right)\nonumber\\
		&\quad\quad\quad- \norm{\nu}zh'(\norm{\nu}z)H_\pi(\nu)\abs{\log(1-\zeta)}\nonumber\\
		&=H_{\pi}(\nu)\left( \left((1-\zeta)^{p} + p(1-\zeta)^{p-1}\zeta\right) \left(    h\left(\frac{\norm{\nu}z}{1-\zeta}\right) - h(\norm{\nu}z) \right) 
		- \norm{\nu}zh'(\norm{\nu}z)\abs{\log(1-\zeta)}\right).
	\end{align}
	We now bound the r.h.s. above. Note that $(1-\zeta)^{p} + p(1-\zeta)^{p-1}\zeta = 1 - \bO(\zeta^2)$ as $\zeta\to0$. Then, using \eqref{eq:zetalogIneq} and then  $\zeta-\frac{\zeta}{1-\zeta}=\frac{\zeta^2}{1-\zeta}$ to obtain
	$\abs{\log(1-\zeta)}=\frac{\zeta}{1-\zeta}+\bO\left(\frac{\zeta^2}{1-\zeta}\right)$;  \eqref{eq:boundGdual} is bounded by
	$$
	C\left(
	h\left(\frac{\norm{\nu}z}{1-\zeta}\right) - h(\norm{\nu}z) - \norm{\nu}z\frac{\zeta}{1-\zeta}h'(\norm{\nu}z) \right).
	$$
	Using Taylor's expansion around $\zeta=0$ (note that $\frac{1}{1-\zeta}-1=\frac{\zeta}{1-\zeta}$) the above is bounded by
	\begin{align*}
		&C\left(
		\frac{h''(\eta_\zeta \norm{\nu}z)}{2}\left(\frac{\norm{\nu}z\zeta}{1-\zeta}\right)^2\right)
		=C\left(
		\frac{(\eta_\zeta\norm{\nu}z)^{2}h''(\eta_\zeta \norm{\nu}z)}{2}\left(\frac{\zeta}{1-\zeta}\right)^2\eta_\zeta^{-2}\right)
	\end{align*}
	where $\eta_\zeta\in [1,1/(1-\zeta)]$, so that $\eta_\zeta^{-2}$ is bounded on $\zeta\in(0,1/2).$
	Using that $x^2h''(x)\in\cbf(\R)$, the  term in the r.h.s. above is bounded by
	$$
	\forall \zeta\in(0,1/2): \quad C\left(C
	\left(\frac{\zeta}{1-\zeta}\right)^2\eta_\zeta^{-2}\right)<C\zeta^2.
	$$\par
	Putting the two bounds together, the one for $2\leq \ell\leq p$ and the one for $\ell\in\{0,1\}$, we conclude:
	\begin{align}\label{eq:GHLambda(0,1/2)}
	\forall \zeta\in(0,1/2): \quad	&\Bigg( \sum_{\ell = 2}^{p}  \sum_{\substack{J\subset [p]\\ \card J=\ell}} (1-\zeta)^{p-\ell}\zeta^{\ell} \lbr    h\left(\frac{\norm{\nu}z}{1-\zeta}\right) H_{\pi^{(J)}}(\nu)
		- h(\norm{\nu}z)  H_{\pi}(\nu)\rbr\nonumber\\
		&- \norm{\nu}zh'(\norm{\nu}z)H_\pi(\nu)\abs{\log(1-\zeta)}\Ind{\zeta\le 1/2}\Bigg)
		\leq C \zeta^2
	\end{align}
	uniformly on $\nu,\pi,z$. 
	\par
	On the other hand,
    since $h\in\cbf(\R)$ and $H_\pi(\nu)\in\bbf(\sM(\typeS))$, we also have
	\begin{align}\label{eq:GHLambda(1/2,1)}
	\forall \zeta\in(1/2,1): \quad		&\sum_{\ell = 0}^{p}  \sum_{\substack{J\subset [p]\\ \card J=\ell}} (1-\zeta)^{p-\ell}\zeta^{\ell} \left(    h\left(\norm{\nu}z\right) (H_{\pi^{(J)}}(\nu)
		- H_{\pi}(\nu))\right)\nonumber\\
		&<\sum_{\ell = 0}^{p}  \sum_{\substack{J\subset [p]\\ \card J=\ell}} (1-\zeta)^{p-\ell}\zeta^{\ell} C<C
	\end{align}
	uniformly on $\nu,\pi,z$.\par
	Combining both bounds we  obtain that if $\Lambda'$ is another finite measure on $(0,1)$, then
	\begin{align}\label{eq:jointGeneratorJumpBound}
		\Bigg\lvert \int_{(0,1)}\frac{\Lambda(d\zeta) - \Lambda'(d\zeta)}{\zeta^2}&\Bigg\{\sum_{\ell = 2}^{p}  \sum_{\substack{J\subset [p]\\ \card J=\ell}} (1-\zeta)^{p-\ell}\zeta^{\ell} \lbr    h\left(\frac{\norm{\nu}z}{1-\zeta}\right) H_{\pi^{(J)}}(\nu)
		- h(\norm{\nu}z)  H_{\pi}(\nu)\rbr\nonumber\\
		&- \norm{\nu}zh'(\norm{\nu}z)H_\pi(\nu)\abs{\log(1-\zeta)}\Ind{\zeta\le 1/2}\Bigg\}\Bigg\rvert 
		\leq C\norm{\Lambda - \Lambda'}_{TV}
	\end{align}
	from which $\norm{\GHall G_{p,\phi,h}- \GHall' G_{p,\phi,h}}_\infty\leq C\norm{\Lambda- \Lambda'}_{TV}$
	follows. \par
	Finally, the first two terms of $\GHall$ in \eqref{eq:jointGenerator} are easily bounded using that 
	 $h(x),xh'(x),x^2h''(x)\in\cbf(\R)$, together with the observation that $\norm{H_\pi}_\infty\leq\norm{\phi}$ for any partition $\pi\in\PS{p}$, $p\geq 1$. Setting $\Lambda'\equiv 0$ in \eqref{eq:jointGeneratorJumpBound} gives a bound for the third term, and $\norm{\GHall G_{p,\phi,h}}_{\infty}<\infty$ follows.	
\end{proof}

The following is an immediate consequence of Lemma \ref{le:dualGeneratorBound} and the definition of $(\Gdall,D_{\Gdall})$ in \eqref{eq:dualGenerator}. This corollary is used, in section \ref{sec:dualConstruction}, to prove that the weak limit of solutions to the martingale problems for $(\Gdall_n, D_\Gall)$ 
with corresponding jump measures $\Lambda_n$, is a solution to the martingale problem for $(\Gall, D_\Gall)$ with jump measure $\Lambda=\lim_{n\to\infty}\Lambda_n$.
 \begin{corollary}\label{cor:convergenceDualGenerators}
 	Let $G\in D_{\Gdall}$. Then, for any pair of finite measures $\Lambda,\Lambda'$ on $(0,1)$ and corresponding operators $\Gdall,\Gdall'$ we have,
 	for some $C>0$,
 	$$
 	\norm{\Gdall G - \Gdall' G}_\infty\leq C\norm{\Lambda- \Lambda'}_{TV}.
 	$$
 \end{corollary}
\section{Duality and consequences}\label{sec:duality}
\begin{lemma}[Operator duality]\label{le:operatorDuality}
	Let $G_{p,\phi,h}\in D_{\GHall}$. Then 
	\begin{equation}\label{eq:dualGeneratorsEqual}
		\Gall G_{p,\phi,h}(\nu,\tilde \pi,z)=\GHall G_{p,\phi,h}(\nu,\tilde\pi,z)= \Gdall  G_{p,\phi,h}(\nu,\tilde\pi,z).
	\end{equation}
	On the left, $\Gall$ is applied to the function $G_{p,\phi,h}(\cdot,\tilde\pi,z)$ and the resulting function is evaluated at $\nu$. On the right, $\Gdall$ is applied to the function $G_{p,\phi,h}(\nu,\cdot,\cdot)$ and the result is then evaluated at $(\tilde\pi,z)$.
\end{lemma}
\begin{proof}
	The equality on the r.h.s. of \eqref{eq:dualGeneratorsEqual} is the definition in \eqref{eq:dualGenerator}.\par
	Let us now prove the equality on the l.h.s. Fix $p\geq1,\pi\in\PS{[p]},z\in\R$ and let $F(\nu)=G_{p,\phi,h}(\nu,\tilde\pi,z)$.
	We first compute
	$\Gjump{\Lambda} F(\nu)$. It is sufficient to consider the case $\restr{\tilde\pi}{p}=\pi=\{\{1\},\cdots,\{p\}\}$, which simplifies notation. Indeed, otherwise one only needs to replace $\phi$ by $\phi_\pi$ in the following computations. Note that
	\begin{align}\label{eq:firstDerivativeDual}
		F'(\nu;a) &= \lim_{\eps\downarrow0} \frac{F(\nu+\eps\delta_a)-F(\nu)}{\eps}
		=\lim_{\eps\downarrow0} \frac{h((\norm{\nu}+\eps)z)H_\phi(\nu+\eps\delta_a) - h(\norm{\nu}z)H_\pi(\nu)}{\eps}\nonumber\\
		&= zh'(\norm{\nu}z) H_\pi(\nu) + h(\norm{\nu}z)H_\pi'(\nu;a).
	\end{align}\par
	Let us compute $H_\pi'(\nu;a)$ using \eqref{eq:pointDerPolyn} and --formally-- the product rule of differentiation. We have
	\begin{align}\label{eq:freqGeneratorIsFV0}
		H_\pi'(\nu;a)&= \frac{1}{\norm{\nu}^{2p}}\left(
		\sum_{i=1}^p \lbr \mint{\phi}{\nu^{\otimes i}\delta_a\nu^{\otimes p-i-1}}\norm{\nu}^p \rbr - 
		p\norm{\nu}^{p-1}\mint{\phi}{\nu^{\otimes p}}
		\right)\nonumber\\
		&=\frac{1}{\norm{\nu}}\sum_{i=1}^p \lbr \frac{\mint{\phi}{\nu^{\otimes i}\delta_a\nu^{\otimes p-i-1}}}{\norm{\nu}^{p-1}}
		-\frac{\mint{\phi}{\nu^{\otimes p}}}{\norm{\nu}^p} \rbr.
	\end{align}
	The above implies
	$\int_\typeS \nu(da)h(\norm{\nu}z)H_\pi'(\nu;a)=0$ so that, from \eqref{eq:firstDerivativeDual},
	\begin{equation}\label{eq:GdualCompensator}
		\int_\typeS \nu(da) F'(\nu;a) = \int_\typeS \nu(da) zh'(\norm{\nu}z)H_\pi(\nu)=
		\norm{\nu}zh'(\norm{\nu}z)H_\pi(\nu).
	\end{equation}
	On the other hand,
	\begin{align*}
		&F(\nu + \norm{\nu}\frac{\zeta}{1-\zeta}\delta_a) -F(\nu)\nonumber\\
		&= 
		\sum_{\ell = 0}^{p} \sum_{\vec m\in \perms(a,\ell,\nu,p)}h\left(\frac{\norm{\nu }z}{1-\zeta}\right) \frac{\mint{\phi}{\otimes_{i=1}^p m_i }}{\norm{\nu}^{p}(1-\zeta)^{-p}}
		\left(\frac{\norm{\nu}\zeta}{1-\zeta}\right)^{\ell} - \sum_{\ell = 0}^{p}(1-\zeta)^{p-\ell}\zeta^{\ell} \sum_{\vec m\in \perms(a,\ell,\nu,p)} h(\norm{\nu}z)\frac{\mint{\phi}{\nu^{\otimes p}}}{\norm{\nu}^p} \nonumber\\
		&=\sum_{\ell = 0}^{p} (1-\zeta)^{p-\ell}\zeta^{\ell} \sum_{\vec m\in \perms(a,\ell,\nu,p)} \lbr h\left(\frac{\norm{\nu}z}{1-\zeta}\right)\frac{\mint{\phi}{\otimes_{i=1}^p m_i }}{\norm{\nu}^{p-\ell}}
		- h(\norm{\nu}z) \frac{\mint{\phi}{\nu^{\otimes p}}}{\norm{\nu}^p}\rbr.
	\end{align*}
	Integrating the r.h.s. above  with respect to $\nu(da)/\norm{\nu}$, we obtain
	
	\begin{align}\label{eq:GdualJumps}
		&\int_\typeS \frac{\nu(da)}{\norm{\nu}}\lbr F\left(\nu + \norm{\nu}\frac{\zeta}{1-\zeta} \delta_a\right)
		- F(\nu) \rbr\nonumber\\
		&=\sum_{\ell = 0}^{p} (1-\zeta)^{p-\ell}\zeta^{\ell}  \sum_{\substack{J\subset [p]\\ \card J=\ell}} \lbr  h\left(\frac{\norm{\nu}z}{1-\zeta}\right)  \frac{\mint{\phi_{\pi^{(J)}}}{\nu^{\otimes (p-\ell+1)}}}{\norm{\nu}^{p-\ell+1}}
		- h(\norm{\nu}z) \frac{\mint{\phi}{\nu^{\otimes p}}}{\norm{\nu}^p}\rbr\nonumber\\
		&=\sum_{\ell = 0}^{p} (1-\zeta)^{p-\ell}\zeta^{\ell}  \sum_{\substack{J\subset [p]\\ \card J=\ell}} \lbr  h\left(\frac{\norm{\nu}z}{1-\zeta}\right)  \frac{\mint{\phi_{\pi^{(J)}}}{\nu^{\otimes p}}}{\norm{\nu}^{p}}
		- h(\norm{\nu}z) \frac{\mint{\phi}{\nu^{\otimes p}}}{\norm{\nu}^p}\rbr\nonumber\\
		&=\sum_{\ell = 0}^{p} (1-\zeta)^{p-\ell}\zeta^{\ell}  \sum_{\substack{J\subset [p]\\ \card J=\ell}} \lbr  h\left(\frac{\norm{\nu}z}{1-\zeta}\right) H_{\pi^{(J)}}(\nu)
		- h(\norm{\nu}z)  H_{\pi}(\nu)\rbr.
	\end{align}
	Combining \eqref{eq:GdualCompensator} and \eqref{eq:GdualJumps} we obtain
	\begin{align}
		&(\Gdrift{\drift}+\Gjump{\Lambda})F(\nu)= \drift \norm{\nu}zh'(\norm{\nu}z)H_\pi(\nu)\nonumber\\
		&+\int_{(0,1)}\frac{\Lambda(d\zeta)}{\zeta^2}\bigg\{\left(\sum_{\ell = 0}^{p} (1-\zeta)^{p-\ell}\zeta^{\ell}  \sum_{\substack{J\subset [p]\\ \card J=\ell}} \lbr  h\left(\frac{\norm{\nu}z}{1-\zeta}\right) H_{\pi^{(J)}}(\nu)
		- h(\norm{\nu}z)  H_{\pi}(\nu)\rbr\right)\nonumber\\
		&\quad \quad- \norm{\nu}zh'(\norm{\nu}z)H_\pi(\nu)\abs{\log(1-\zeta)}\Ind{\zeta\le 1/2}\bigg\}.
	\end{align}\par
	We now compute $\Gdiff{\sigma}F$. For this let us first compute $H''_\pi(\nu;a,a)$. 
	Fix $a\in\typeS$ and  note that, for any $1\leq i\leq p$, 
	we have $\frac{\mint{\phi}{\nu^{\otimes i}\delta_a\nu^{\otimes p-i-1}}}{\norm{\nu}^{p-1}}=\frac{\mint{\phi^{(i)}}{\nu^{\otimes i}\delta_a\nu^{\otimes p}}}{\norm{\nu}^{p}}=H_\pi^{\phi^{(i)}}(\nu)$
	where $\phi^{(i)}(x_1,\cdots,x_{p})=\phi(x_1,\cdots,x_{i-1},a,x_{i+1},...,x_{p}).$ Then, 
	plugging in \eqref{eq:freqGeneratorIsFV0} we obtain
	\begin{equation}\label{eq:H'bound}
	H_\pi'(\nu;a)=\frac{1}{\norm{\nu}}\sum_{i=1}^p \lbr H_\pi^{\phi^{(i)}}(\nu)
	-H_\pi(\nu) \rbr.
	\end{equation}
	Taking the derivative of each term above --using the product rule of differentiation as before-- we obtain
	\begin{equation}\label{eq:H''bound}
	H''_\pi(\nu;a,a)=\sum_{i=1}^p \frac{(\norm{\nu}\left(H_\pi^{\phi^{(i)}}\right)'(\nu;a)-H_\pi^{\phi^{(i)}}(\nu)) - (\norm{\nu}H_\pi'(\nu;a) - H_\pi(\nu))}{\norm{\nu}^2}.
	\end{equation}
	As before, from \eqref{eq:freqGeneratorIsFV0} we obtain
	$
	\int_\typeS \nu(da) \sum_{i=1}^p  \frac{H_\pi^{\phi^{(i)}}(\nu)  - H_\pi(\nu)}{\norm{\nu}^2} = 0,
	$
	and also
	$
	\int_\typeS \nu(da) \sum_{i=1}^p \frac{H'_\pi(\nu;a)}{\norm{\nu}}=0.
	$
	Thus, using \eqref{eq:freqGeneratorIsFV0} --replacing $\phi$ therein by $\phi^{(i)}$-- we obtain
	\begin{align}\label{eq:freqGeneratorIsFV1}
		\norm{\nu}\int_\typeS\nu(da) H''_\pi(\nu;a,a) 
		&=\sum_{i=1}^p \int_\typeS\frac{\nu(da)}{\norm{\nu}}\sum_{j=1}^p \lbr \frac{\mint{\phi^{(i)}}{\nu^{\otimes {j-1}}\delta_a\nu^{\otimes p-{j}}}}{\norm{\nu}^{p-1}}
		-\frac{\mint{\phi^{(i)}}{\nu^{\otimes p}}}{\norm{\nu}^p} \rbr \nonumber\\
		&=2\sum_{\substack{J\subset [p]\\ \card J=2}} \lbr   H_{\pi^{(J)}}(\nu)
		-   H_{\pi}(\nu)\rbr.
	\end{align}\par
	Continuing with the computation of $\Gdiff{\sigma}F$, from \eqref{eq:firstDerivativeDual} we obtain 
	\begin{equation*}\label{eq:F''bound}
	F''(\nu;a,a) = z^2h''(\norm{\nu} z)H_\pi(\nu) + 2zh'(\norm{\nu} z)H_\pi'(\nu;a) 
	+ h(\norm{\nu}z)H_\pi''(\nu;a,a).
	\end{equation*}
	Note that, similarly as before, $\norm{\nu}\int_\typeS\nu(da) zh'(\norm{\nu} z)H_\pi'(\nu;a)=0$ so that,
	using \eqref{eq:freqGeneratorIsFV1} in the second term of the r.h.s. above, 
	\begin{equation}\label{eq:Gdiffeq1}
	\Gdiff{\sigma}F=\norm{\nu}\int_\typeS\nu(da) \frac{\sigma^2}{2} F''(\nu;a,a) = \frac{\sigma^2}{2}\norm{\nu}^2z^2h''(\norm{\nu} z)H_\pi(\nu) +  \sigma^2{h(\norm{\nu}z)}\sum_{\substack{J\subset [p]\\ \card J=2}} \lbr   H_{\pi^{(J)}}(\nu)
	-   H_{\pi}(\nu)\rbr.
	\end{equation}
	Gathering terms we obtain \eqref{eq:dualGeneratorsEqual}.
\end{proof}

We now make use of the computations carried out in the above proof in order to prove that $D_\DWG'\subset \D(\DWG)$ (Proposition \ref{prop:DWG'subsetDDWG}). This in turn is one of the steps of the proof
that $D_\Gall'\subset D_\Gall$ in the following Corollary \ref{cor:D_Gall'<D_Gall}.
\begin{proof}[Proof of Proposition \ref{prop:DWG'subsetDDWG}]
	Let $\phi\in \cbf(\typeS^p)$ in \eqref{eq:defD_DWG'} be of the form $\phi=\phi_1\dots \phi_p$ with $\phi_i\in\cbf(\typeS)$. Also consider $\phi_0\equiv 1$. Assume that $\norm\mu>0$. Then,
	for $h\in D_h$ for which $\restr{h}{(0,\infty)}\in \cKdf{\infty}(\Rp)$, 
	clearly $F(\mu)=h(\norm{\mu})\mint{\phi}{\mu^{\otimes p}}=h(\mint{\phi_0}{\mu})\prod_{i=1}^p \mint{\phi_i}{\mu}$ 
	satisfies $F\in D_\DWG$. The latter easily extends to $F(\mu)=h(\norm{\mu})\mint{\phi}{\mu^{\otimes p}}$ for any $\phi\in\break\linearspan(\{\phi_1\cdots \phi_p\colon \phi_i\in\cbf(\typeS)\})$. 
	Note that, taking $\hat h(x)=h(x)x^{p}$ we have, for $\hat F(\mu)=\hat h(\norm{\mu}) H^{(\phi)}_{\pi}(\mu)$ with $\pi$ the singleton partition of $[p]$,
	that $\hat F(\mu)= F(\mu)$. From \eqref{eq:mvpSMHdiffgenerator} and \eqref{eq:Gdiffeq1} we thus obtain, for $\norm{\mu}>0$,
	\begin{align*}
		\DWG F(\mu)& = \frac{1}{\norm{\mu}} \Gdiff{\sigma} \hat F(\mu)	   	\nonumber\\
		&= \frac{\sigma^2}{2}\norm{\mu}\hat h''(\norm{\mu})H_\pi(\mu) +  \sigma^2\frac{\hat h(\norm{\mu})}{\norm{\mu}}\sum_{\substack{J\subset [p]\\ \card J=2}} \lbr   H^{(\phi)}_{\pi^{(J)}}(\mu)
		-   H^{(\phi)}_{\pi}(\mu)\rbr
	   	\nonumber\\
	    &=\frac{\sigma^2}{2}(h''(\norm{\mu})\norm{\mu}^{p+1}+2ph'(\norm{\mu})\norm{\mu}^{p}+(p-1)h(\norm{\mu})\norm{\mu}^{p-1})H_\pi(\mu) 
	    \nonumber\\
	    &+  \sigma^2h(\norm{\mu})\norm{\mu}^{p-1}\sum_{\substack{J\subset [p]\\ \card J=2}} \lbr  
	    H^{(\phi)}_{\pi^{(J)}}(\mu)
	    -   H^{(\phi)}_{\pi}(\mu)\rbr
	\end{align*}
	and $\DWG F(\mu)=0$ if $\mu=0$. Note that $h$ being of compact support, the above expression is uniformly bounded on
	$\mu\in\sM(\typeS)$ by a constant $C$ that depends on $\phi$ only through $\norm{\phi}_\infty$. Furthermore, 
	if $\phi_n\to\phi$ boundedly point-wise, then by dominated convergence also $H^{(\phi_n)}_\pi\to H^{(\phi)}_\pi$ boundedly pointwise in $\mu\in\sM(\typeS)\setminus \{0\}.$ 
	Then for the corresponding functions $F_{\phi_n}$, we have $F_{\phi_n}\to F_\phi$ and $\DWG F_{\phi_n}\to \DWG F_\phi$  boundedly point-wise. Note that
	by the Stone-Weierstrass theorem the set $\linearspan(\{\phi_1\cdots \phi_p\colon \phi_i\in\cbf(\typeS)\})$ is dense in $\czf(\typeS^p)$ for the topology of uniform convergence
	on compact sets of $\typeS^p$, and thus its bounded point-wise closure contains $\cbf(\typeS^p)$. Thus the set of functions $D''_{\DWG}=\{h(\norm{\mu})\mint{\phi}{\mu^{\otimes p}}\colon h\in D_h, p\geq1, \phi\in\cbf(\typeS^p)\}$
	satisfies that $(\DWG, D''_\DWG)$ is in the bounded point-wise closure of $(\DWG, D_\DWG)$, so that
	any solution to the martingale problem for $(\DWG,D_\DWG)$ is also a solution to the martingale problem for
	$(\DWG,D''_\DWG)$ (see e.g. Section III in Chapter IV \cite{EthierKurtz86}). Thus we conclude $D''_\DWG\subset \D(\DWG)$.
	Clearly for any $h\in D_h$ we have
	$x^{-p}h(x)\in D_h$ for any $p\geq1$ so that also $D'_\DWG\subset D''_{\DWG}$, and the proof is finished.
\end{proof}

\begin{corollary}\label{cor:D_Gall'<D_Gall}
	We have $D_{\Gall}'\subset D_{\Gall}$.
\end{corollary}
\begin{proof}
	Let us carry the notation from Lemma \ref{le:operatorDuality}, in particular recall the function $F$ therein. 
	In the following $C>0$ will denote a constant that may vary from line to line but that does not depend on $\nu$. 
	Clearly, if $h\in D_h$ then $h$ is bounded, and so is $F$. 
	From \eqref{eq:H'bound} we have $\norm{\nu}H_\pi'(\nu)<C$. 
	From the latter together with \eqref{eq:firstDerivativeDual}, and using that $xh'(x)<C$ whenever $h\in D_h$, we obtain $\norm{\nu}F'(\nu)<C.$
    Also, 
	from $\norm{\nu}H_\pi'(\nu)<C$ together with \eqref{eq:H''bound} we obtain
	$\norm{\nu}^2H_\pi''(\nu)<C$. From the latter, together with \eqref{eq:F''bound} and  for $h\in D_h$, we also obtain $\norm{\nu}F''(\nu)<C$. This proves
	$\forall a\in\typeS, \nu\in\sM(\typeS);\quad  \abs{F(\nu)} + \norm{\nu}\abs{F'(\nu;a)} + \norm{\nu}^2\abs{F''(\nu;a,a)}
	\leq C.$ Finally, since by Proposition \ref{prop:DWG'subsetDDWG} we also have $F\in \D(\DWG)$, we conclude $F\in D_\Gall.$
\end{proof}

In order to ensure uniqueness of solutions to the martingale problems for $(\Gall,D'_{\Gall})$ and $(\Gdall,D_{\Gdall})$, we need to show that the duality in Theorem \ref{th:duality} (eq. \eqref{eq:duality}) holds for a sufficiently large class of functions. This is a consequence of the following lemma.
\begin{lemma}\label{le:dualityDomainSeparates}
	\begin{enumerate}   
		\item \label{itTh:D_Gall'separates}
    The set $\linearspan(D_{\Gall}')$  is dense in $\cbf(\sM(\typeS))$ in the topology of uniform convergence on compact sets, and in particular is separating on $\sPM(\sM(\typeS))$.\par
    	\item  \label{itTh:D_Gdallseparates}
	The set $\linearspan(D_{\Gdall})$ is
	separating on $\sPM(\PS{\infty}\times \Rp).$
	\end{enumerate}
\end{lemma}
\begin{proof}
	Both statements follow from an application of the Stone-Weierstrass approximation theorem.\par 
	To obtain {\it \ref{itTh:D_Gall'separates}} we will prove that the set
	$D_{\Gall}'$ is closed under multiplication and separates points in $\sM(\typeS)$; this will of course imply the same properties for $\linearspan(D_{\Gall}')$. Having proved the latter, the Stone-Weierstrass theorem then implies that the set $\linearspan(D_{\Gall}')$ is dense in $\cbf(\sM(\typeS))$ in the topology of uniform convergence on compact sets, and thus separating on $\sPM(\sM(\typeS))$.\par
	We first show that the sets of functions $D_{\Gall}'$ is closed under multiplication. Indeed, let $F^{(i)}$, $i\in\{1,2\}$, be defined as in \eqref{eq:GDualityDomain} with corresponding parameters 
	$p^{(i)}, \phi^{(i)}, h^{(i)}, \tilde \pi^{(i)}, z^{(i)}$. Set $\pi^{(i)}=\restr{\tilde\pi^{(i)}}{p^{(i)}}.$ Let $\pi^{(2)}+p^{(1)}$ be the partition of $\PS{\{p^{(1)}+1,\cdots,p^{(1)}+p^{(2)}\}}$ that
	results from translating each element $i$ of the blocks of $\pi^{(2)}$ by $i\to i+p^{(1)}$.  Construct $\pi\in\PS{[p^{(1)}+p^{(2)}]}$ by taking the union
	$\pi^{(1)} \cup (\pi^{(2)}+p^{(1)}).$ Let $\tilde\pi\in \PS{\infty}$ be any partition such that $\restr{\tilde\pi^{(i)}}{p^{(1)}+p^{(2)}}=\pi.$ Also set
	$\phi(a_1,\cdots,a_{p^{(1)}+p^{(2)}})=\phi^{(1)}(a_1,\cdots,a_{p^{(1)}})\phi^{(2)}(a_{p^{(1)}+1},\cdots,a_{p^{(1)}+p^{(2)}}).$ Let also $\phi_{i,\pi^{(i)}}$ be the function 
	defined by \eqref{eq:defPhiPi} with $\phi$ and $\pi$ therein set to $\phi=\phi^{(i)}$ and $\pi=\pi^{(i)}$. Then
	$$
	H_{\pi}(\nu)=\mint{\phi_\pi}{\left(\frac{\nu}{\norm{\nu}}\right)^{\otimes \card \pi}} = 
	\mint{\phi_{1,\pi^{(1)}}}{\left(\frac{\nu}{\norm{\nu}}\right)^{\otimes \card \pi^{(1)}}} 
	\mint{\phi_{2,\pi^{(2)}}}{\left(\frac{\nu}{\norm{\nu}}\right)^{\otimes \card \pi^{(2)}}}=H_{\pi^{(1)}}(\nu)H_{\pi^{(2)}}(\nu).
	$$
    Also, setting $h(x)=h^{(1)}(x z^{(1)})h^{(2)}(x z^{(2)})$ and $z=1$, we have $h(\norm{\nu}z)=h^{(1)}(\nu  z^{(1)})h^{(2)}(\nu  z^{(2)})$. Differentiating $h$ using the product rule, and using that $h^{(1)},h^{(2)}\in D_h$  one easily sees that also $h\in D_h$.
    Thus
    $F^{(1)}(\nu)F^{(2)}(\nu)=h(\nu z)H_{\pi}(\nu)\in D_{\Gdall}'.$\par 
    That $D_{\Gall}'$ separates points is a simple consequence of the fact $D_h$ separates points in $\Rp$, and that
    the family of functions
    $\{\rho \to \mint{\phi}{\rho}\}_{\phi\in\bbf(\typeS)}$ separates points in $\sPM(\typeS).$ This concludes the proof of {\it\ref{itTh:D_Gall'separates}}. This completes de proof of {\it\ref{itTh:D_Gall'separates}}.
    \par
    We now prove {\it \ref{itTh:D_Gdallseparates}}. Let $p\geq1$ and $\hat \pi \in \PS{[p]}$ be arbitrary but fixed. Consider the function $\phi^{(\hat \pi)}\in \bbf(\typeS^p)$ such that  $\phi^{(\hat \pi)}(a_1,\cdots,a_p)=1$ if one obtains the partition $\hat\pi$ after putting together in the same block the indices $i$ and $j$  if and only if $a_i=a_j$. Set $\phi^{(\hat \pi)}(a_1,\cdots,a_p)=0$ whenever the partition obtained in this way is different from $\hat\pi$. Formally, for $1\leq i\leq \card \hat\pi$, let $\min \hat \pi_i$ be the smallest element in $\hat\pi_i$. Also 
    recall that $\hat \pi(j)=i$ whenever $j\in\hat \pi_i$. Then set
    $$
    \phi^{(\hat \pi)}(a_1,\cdots, a_p)= \left(\prod_{i=1}^{\card\hat\pi} \prod_{\substack{j\in[p]\\\hat\pi(j)=i}} \Ind{a_j=a_{\min\hat\pi_i}}\right)
    \Ind{a_{\min\hat\pi_1}\neq \cdots \neq a_{\min\hat\pi_{\card\hat\pi}}}.
    $$
    Furthermore, a simple computation
    shows that for all $\pi=\restr{\tilde\pi}{p},$ $\tilde{\pi}\in\PS{\infty}$, the function $ \phi^{(\hat\pi)}_\pi$ is given by
    $$
    \phi^{(\hat\pi)}_\pi(a_1,\cdots,a_{\card\pi }) \equiv  (\Ind{\pi = \hat\pi}) (\Ind{a_{\min\pi_1}\neq \cdots \neq a_{\min\pi_{\card\pi}}}) = (\Ind{\pi = \hat\pi})( \Ind{a_{\min\hat\pi_1}\neq \cdots \neq a_{\min\hat\pi_{\card\hat\pi}}}).
    $$
    Let $\rho_{\hat\pi}= (\card\hat\pi)^{-1} \sum_{i=1}^{\card\hat\pi}\delta_{b_i}\in \sPM(\typeS)$, where $(b_i)_{i=1}^{\card\hat\pi}$ are  arbitrary but distinct elements  
    in $\typeS$. Then, letting $c_{\hat\pi}$ be the constant $c_{\hat\pi}=\mint{\Ind{a_{\min\hat\pi_1}\neq \cdots \neq a_{\min\hat\pi_{\card\hat\pi}}}}{\rho^{\otimes p}_{\hat\pi}}=\prod_{i=1}^{\card\hat\pi}\frac{i}{\card\hat\pi}$
    we obtain
    $$
    \mint{\phi^{(\hat\pi)}_\pi}{\rho^{\otimes p}_{\hat\pi}} = c_{\hat\pi}\Ind{\pi=\hat\pi} 
    .
    $$
    Consider the set of functions $D_{\Gdall}'$ on $\PS{\infty}\times\Rp$ given by
    $$
    D_{\Gdall}'\coloneqq \left\{G(\tilde\pi,z)=h(z)\frac{1}{c_{\hat\pi}} \mint{\phi^{(\hat\pi)}_\pi}{\rho^{\otimes p}_{\hat\pi}}
    \colon h\in D_h, p\geq 1, \hat\pi\in\PS{[p]} \right\} \equiv 
    \left\{h( z) \Ind{\pi=\hat\pi}\colon h\in D_h, p\geq 1, \hat\pi\in\PS{[p]}  \right\},
    $$
    where we recall that $\pi=\restr{\tilde\pi}{p}$. 
    Note that the indicator function $\Ind A$ of diagonal (closed) sets  $A\subset\typeS^k$ can be boundedly point-wise approximated
    by functions in $\cbf(\typeS^k)$, e.g. by $\phi_n(a_1,\cdots,a_p)=e^{-n \dist_{\typeS^p}((a_1,\cdots,a_p), A)}$ where $\dist_{\typeS^p}$ is the natural metric on $\typeS^p$.  This implies that
    $\phi^{(\hat\pi)}$ can be  boundedly approximated by functions $\phi$ in $\cbf(\typeS^p)$ and, by dominated convergence,
    also the function $\tilde\pi\to\mint{\phi^{(\hat\pi)}_{\restr{\tilde\pi}{p}}}{\rho^{\otimes p}_{\hat\pi}}$ can be boundedly point-wise approximated. We thus conclude that $D_{\Gdall}'$ is in the bounded point-wise closure
    of $D_{\Gdall}$. Therefore, a dominated convergence argument yields that if $\linearspan(D_{\GHall'})$ is separating on $\sPM(\PS{\infty}\times \Rp)$, then so is $\linearspan(D_{\Gdall})$.
    Let us then prove that $D_{\Gdall}'\subset \cbf(\PS{\infty}\times \Rp)$ and that it is closed under multiplication and separates points in $\PS{\infty}\times\Rp$. Similarly
    as before, via an application of the Stone-Weierstrass theorem,  this will imply that $\linearspan(D_{\Gdall}')$ is separating on $\sPM(\PS{\infty}\times \Rp)$.\par 
    To prove that $D_{\Gdall}'\subset \cbf(\PS{\infty}\times \Rp)$ we need only show that, for any $\hat\pi\in \PS{[p]}$, the  function $\Ind{\tilde\pi\vert_p = \hat\pi}$ is a continuous
    function of $\tilde\pi \in \PS{\infty}.$ Indeed, note that $\Ind{\tilde\pi\vert_p = \hat\pi}\equiv \Ind{\tilde\pi'\vert_p = \hat\pi}$ whenever
    $\tilde\pi\vert_p=\tilde\pi'\vert_p $; i.e. whenever $\tilde\pi'$ is in the $\PS{\infty}$-ball of radious $1/p$ around $\tilde \pi.$ (see Lemma 2.6 \cite{Bertoin2006}).\par
    That $D_{\Gdall}'$ separates points in $\PS{\infty}\times\Rp$ is a simple consequence of the fact that $D_h$ separates
    points in $\Rp$, and that the family of functions $\{\Ind{\tilde \pi\vert_p=\hat \pi}\colon p\geq 1, \hat\pi\in\PS{[p]}\}$ separates points in
    $\PS{\infty}$. We now prove that $D_{\Gdall}'$ is closed under multiplication. 
    Indeed, for $i\in\{1,2\}$ let $G^{(i)}\in D_{\Gdall}'$ be given by  $p^{(i)}\geq1, \hat \pi^{(i)}\in \PS{[p^{(i)}]}$, and $h^{(i)}\in D_h$. 
    Without loss of generality let us assume 
    $p^{(1)}\leq p^{(2)}$. Note that if $\restr{\hat\pi^{(2)}}{p^{(1)}}\neq \hat \pi^{(1)}$ then  $G^{(1)}G^{(2)}\equiv 0\in D_{\Gdall}'$. On the other hand, if $\restr{\pi^{(2)}}{p^{(1)}}= \pi^{(1)}$ then $(\Ind{\tilde\pi\vert_{p^{(1)}} = \hat\pi^{(1)}})(\Ind{\tilde\pi\vert_{p^{(2)}} = \hat\pi^{(2)}})=\Ind{\tilde\pi\vert_{p^{(2)}} = \hat\pi^{(2)}}$.
    In this case $G^{(1)}G^{(2)}=h^{(1)}(z)h^{(2)}(z)\Ind{\tilde\pi\vert_{p^{(2)}} = \hat\pi^{(2)}}$. Clearly $h^{(1)}h^{(2)}\in D_h$, so that $G^{(1)}G^{(2)}\in D_{\Gdall}'$. This concludes the proof. 
\end{proof}

For the proof of Theorem \ref{th:duality}, in particular for the stated uniqueness of solutions, we assume existence of solutions to the martingale
problems for both $(\Gall,D_{\Gall}')$ and $(\Gdall,D_{\Gdall}).$ The latter are proved using independent arguments in the first steps of the proofs of Theorems \ref{th:MVPsmhExist} and
\ref{th:FellerDual} in sections \ref{sec:nucons} and \ref{sec:dualConstruction} respectively. 
\begin{proof}[Proof of Theorem \ref{th:duality}]
	We prove \eqref{eq:duality} via an application of Theorem IV.4.11 in \cite{EthierKurtz86} (with $\alpha$ and $\beta$ therein set to $0$). Given \eqref{eq:dualGeneratorsEqual}, it only remains to 
	verify the $\Lp{1}$-boundedness conditions therein, which in our case will be verified once we prove the following:
	\begin{align}\label{eq:dualityLp1Conditions}
		\forall T\geq0;\quad    &\E\left[\sup_{0\leq s\leq t\leq T} \abs{G_{p,\phi,h}\left(\nu_s,\Pi_t,Z_t\right)}\right]<\infty,\\
		&\E\left[\sup_{0\leq s\leq t\leq T} \abs{\Gall G_{p,\phi,h}\left(\nu_s,\Pi_t,Z_t\right)}\right]<\infty, \nonumber\\
		&\E\left[\sup_{0\leq s\leq t\leq T} \abs{\Gdall G_{p,\phi,h}\left(\nu_s,\Pi_t,Z_t\right)}\right]<\infty, \nonumber
	\end{align}
	where $\stp{\nu_t}$ and $\stp{\Pi_t,Z_t}$ are independent solutions to the martingale problems for $(\Gall,D_{\Gall}')$ and $(\Gdall,D_{\Gdall})$ respectively. In fact
	\eqref{eq:dualityLp1Conditions} follows from the fact that $G_{p,\phi,h},\Gall G_{p,\phi,h},$ and $\Gdall G_{p,\phi,h}$ are all uniformly bounded. This is a consequence of Lemmas \ref{le:dualGeneratorBound} and \ref{le:operatorDuality}.  This proves \eqref{eq:duality}.\par
	Assumming existence of solutions to the martingale
	problems for both $(\Gall,D_{\Gall}')$ and $(\Gdall,D_{\Gdall})$, which is proved using independent arguments in  sections \ref{sec:nucons} and \ref{sec:dualConstruction} respectively.
	The stated uniqueness of solutions to the martingale problems for $(\Gall,D_{\Gall}')$ and $(\Gdall,D_\Gdall)$ both follow from \eqref{eq:duality} and  Lemma \ref{le:dualityDomainSeparates}, which provide the remaining necessary conditions for Proposition IV.4.7. in \cite{EthierKurtz86} to be applied in both cases.
\end{proof}



\section{Construction and Feller property of the SMH process}\label{sec:nucons}
We start this section with the Poissonian construction of the process $\stp{\nu_t}$ when the measure $\zeta^{-2}\Lambda(d\zeta)$ is finite.
A simple way to think
of this process
is as a time-changed and mass-scaled Dawson-Watanabe process
to which atoms are added at times of an independent Poison point process. The size
of such atoms will in general depend on the total population size at jump times.

Formally, for an initial condition $\nu$, let $\stp{\hat\nu_t}=\stp{\mu^{(DW)}_{c_1(t)}}$ 
where $\stp{\mu^{(DW)}_t}$ is a standard $\sigma$-Dawson-Watanabe process started at $\nu$;
and
$c_1(t) 
=\inf\lbr s\geq 0 \colon \int_0^s \frac{1}{\norm{\mu^{(DW)}_u}}du \geq t\rbr$
is the 1-SS Lamperti time change.
We will work with
several independent copies of $\stp{\hat{\nu}_t}$ started at different initial measures $\nu$;
let us denote them by $\stp{\hat{\nu}_t(\nu)}$ when necessary.\par
Let 
\begin{equation*}
\hat{\drift} \coloneqq \drift - \int_{(0,1/2]}\abs{\log(1-\zeta)}\frac{\Lambda(d\zeta)}{\zeta^2}
\end{equation*}
and let
 $\PPP$ be a PPP with intensity $\dif{t}\otimes \zeta^{-2}{\Lambda{(\dif{\zeta})}}$
 on $\Rp\times (0,1)$. For $t>0$
let 
$(t_1,\zeta_1),\dots,(t_K,\zeta_k)$ be the atoms of $\PPP$ such that their first coordinate is less than $t$,
ordered increasingly along the first (time) coordinate and
set $t_{K+1}=t$.  Then,
conditionally on $(t_i, \zeta_i)_{1\leq i \leq K}$, define $\nu_s$, $s\in[0,t)$,
recursively as 
\begin{align}\label{eq:nuFiniteDef}
\nu_s\coloneqq\begin{cases}
e^{\hat{\drift}s} \hat{\nu}_{s}\left(\nu\right) & \text{ if }0\leq s < t_1, \\
e^{\hat{\drift}( s-t_i)} \hat{\nu}_{s-t_i}\left(\nu_{t_i-} + \norm{\nu_{t_i-}}\frac{\zeta_i}{1-\zeta_i}\delta_{a_i}\right) & \text{ if } t_i \leq s < t_{i+1},
\end{cases}
\end{align}
{where the 
the locations of the new atoms $(a_i)_{1\leq i\leq K}$ are chosen independently 
according to $\nu_{t_{i}-}/\norm{\nu_{t_{i}-}}$ and, conditionally on its starting point, the trajectory of
$\left(\hat{\nu}_t\left(\nu_{t_i-} + \norm{\nu_{t_i-}}\frac{\zeta_i}{1-\zeta_i}\delta_{a_i}\right) \right)_{t\geq 0}$ 
is independent from  $(\nu_t)_{0\leq t \leq t_i}$. }\par 

\begin{lemma}\label{le:generatorLMVPfiniteCase}
The process $\stp{\nu_t}$ is {Markov}  with generator 
 $\mathcal{G}$ of the form $\eqref{eq:generatormvMp}$ on $D_{\Gdiff{\sigma}}$.
\end{lemma}

\begin{proof}
First, when $\Lambda\equiv 0$, i.e. for the Markov process $\stp{e^{\hat\drift t}\hat\nu_t}$, a simple
computation yields 
$$
\der{t}\bigg\vert_{t=0}\E_{\nu}[F(e^{\hat{\drift}t}\hat{\nu}_t)]  = \Gdrift{\hat\drift}F(\nu) +  \Gdiff{\sigma}F(\nu).
$$

To add the jumps, and to compute the resulting generator, we use Theorem 2.4 in
\cite{Sawyer1970}.
For this, we set
(in their notation) $\phi_t=C_\Lambda t$ with $C_\Lambda\coloneqq \int_{[0,1]}\zeta^{-2}\Lambda(d\zeta)$ and the transition kernel $K(\nu,\cdot)$ 
given by, for $F\in\bbf(\sM(\typeS))$,
$$
K(\nu, F)=\begin{cases}
\int_{{(0,1)}}\zeta^{-2} \frac{\Lambda(d\zeta)}{C_\Lambda} \int_\typeS \frac{\nu}{\norm{\nu}}(da) F(\nu + \norm{\nu}\frac{\zeta}{1-\zeta}\delta_{a}) & \text{ whenever }\nu\neq 0,\\
F(0) & \text{ otherwise.}
\end{cases}
$$ 
Then, $\phi_t$ being of \say{Kac type}, according to Theorem 2.4 (and Example 1) in \cite{Sawyer1970}, 
the generator of the process in \eqref{eq:nuFiniteDef}
is given by
\begin{align*}
\der{t}\bigg\vert_{t=0}\E[F(e^{\hat{\drift}t}\hat{\nu}_t)]  &= \Gdrift{\hat\drift}F(\nu) +  \Gdiff{\sigma}F(\nu) \\
&+\int_{{(0,1)}}\zeta^{-2} \Lambda(d\zeta)\int_\typeS \frac{\nu}{\norm{\nu}}(da) \left\{F(\nu + \norm{\nu}\frac{\zeta}{1-\zeta}\delta_{a})
 - F(\nu)\right\}\\
 &=\Gdrift{\drift}F(\nu) +  \Gdiff{\sigma}F(\nu) + \Gjump{\Lambda}F(\nu)
\end{align*}
with domain $\D(\Gdiff{\sigma})$.
\end{proof}



Let us now consider the construction when the measure $\zeta^{-2}\Lambda(d\zeta)$ is infinite.
We choose a particular sequence of processes $\left\{\stp{\nun_t}\right\}_n$
that will weakly approximate a process $\stp{\nu_t}$ with generator $\Gall$ as in \eqref{eq:generatormvMp}, the latter will 
be proven
to be a Feller process. 
In the following Proposition \ref{th:weakLimitMartingale}, and in order to simplify technical arguments, here we only focus on the construction of the process $\stp{\nu_t}$ for which we impose stronger conditions than just the convergence of the generators.

In the following we set $\Lambda(\{0\})=\frac{\sigma^2}{2}$. Also, throughout this section,
we will work with the jumping measure $\Ind{0<\zeta\le 1/2}\Lambda(d\zeta)$ instead of $\Lambda(d\zeta)$ for the
limit process; while
the approximating processes $\stp{\nun_t}$ will have jumping measure 
\begin{equation}\label{Lambdan}
\Lambda_n(d\zeta)=\Ind{1/n<\zeta\le 1/2}\Lambda(d\zeta).
\end{equation} 
Removing the atoms in which $\zeta\ge 1/2$ ensures that the jumps of $\stp{\norm{\nun_t}}$ are bounded,
which in turn 
simplifies the proof of tightness for the family $\left\{\stp{\nun_t}\right\}_{n}$. This imposes
no loss of generality for the construction of the process $\stp{\nu_t}$ with
generator of the general form $\Gall$ since the finitely-many large jumps that
occur with intensity $\Ind{\zeta\geq 1/2}\Lambda(d\zeta)$ can always be \say{added back} via the
same argument as in Lemma \ref{le:generatorLMVPfiniteCase}.\par

\begin{proposition}\label{th:weakLimitMartingale}
Let
$\stp{\nun_t}$ be constructed  with $\Lambda_n$ as in \eqref{Lambdan}
where, furthermore, $\nun_0\distto\nu_0$ and 
$\sup_n\E\left[\norm{\nun_0}^p\right]\break<\infty$ for every $p\geq1$. 
Then there exists a process $\stp{\nu_t}$ such that, as $n\to\infty$,
\begin{equation*}
\stp{\nun_t}\distto \stp{\nu_t}
\end{equation*}
in the Skorohod topology on $D([0,\infty); \sM(\typeS))$.
Furthermore, the process $\stp{\nu_t}$ is a solution to the martingale problem for $(\Gall,D_\Gall)$ and is a Feller process. 
Moreover, $\stp{\nu_t/\norm{\nu_t}}$ is a $(\Lambda+\sigma^2/2\delta_0)$-Fleming-Viot process, 
and $\stp{\log\norm{\nu_t}}$ a Lévy process with triplet $(\driftF+\drift, \sigma, \LevyMeasure)$.
\end{proposition}

\begin{proof}
We proceed in multiple steps: proving existence of solutions
to the $\Gall$-martingale problem through tightness of the 
family of processes $\set{\stp{\nun_t}}_{n}$ in the Skorohod topology for  $D([0,\infty); \sM\left(\typeS\right))$,
and the convergence of the approximating generators  $\Gall_n$ to $\Gall$; 
identifying the limit; and showing that it is Feller. 
The proof
of tightness rests on the well-known Aldous-Rebolledo criterion  (e.g. Theorem 1.17 in \cite{etheridge2000introduction}), whereas the characterization and the properties of the limiting
process will be a consequence of the duality relation described in Theorem \ref{th:duality}.\par
{\bf Step 1: Tightness.}
The proof rests on  Theorem III.9.1 in \cite{EthierKurtz86}.
Let us first prove that
\begin{equation}\label{eq:compactContainmentBalls}
\limsup_{n\geq 1}\P\left(\sup_{0\leq t\leq T} \norm{\nun_t} \geq r\right)  \leq \frac{1}{r}\E\left[\sup_{0\leq t\leq T} \norm{\nun_t}\right] 
\toas{r\to\infty}0
\end{equation}
which implies the compact containment condition (eq. (9.1) therein) for our processes
since,
$\typeS$ being compact, the space $\sM\left(\typeS\right)$ is locally compact. In fact
 the balls of radius $r$,
 $
 \sM_r\left(\typeS\right)=\lbr \mu\in\sM\left(\typeS\right)\colon \mint{1}{\mu}\leq r\rbr,
 $
 are compact (see e.g. Theorem 1.14 in \cite{ZenghuLi2022}). \par
{Using Lemma \ref{le:generatorLMVPfiniteCase} and plugging $F(\nu)=h(\norm{\nu})$ in \eqref{eq:generatormvMp} (take e.g.  $\phi\equiv1$
and $z=1$ in \eqref{eq:dualGeneratorsEqual} and \eqref{eq:dualGenerator})}
we obtain that $\stp{\norm{\nun_t}}=\stp{e^{\xin_t}}$ in distribution where  $\stp{\xin_t}$ is the Lévy process
$$\xin_t=\log\left(\norm{\nun_0}\right)+(\drift+\driftF)t + M^{(n)}_t.
$$
with $ M^{(n)}_t$ defined as $M_t$ in  Lemma \ref{le:LevyExponentialMoments} but replacing $\Lambda$ therein by $\Lambda_n$.
Then \eqref{eq:compactContainmentBalls} is an easy consequence of Lemma \ref{le:LevyExponentialMoments}.\par

In view of Theorem III.9.1 in \cite{EthierKurtz86},
and the fact that the polynomials $\polys(\sM(\typeS))$ are dense in $\cbf(\sM(\typeS))$ for 
the topology of uniform convergence in compact sets; 
it remains to prove that $\left\{\stp{F(\nun_t)}\right\}_{n}$ is relatively
compact for every $F\in \polys(\sM(\typeS))\subset D_\Gall$, where we recall that $D_\Gall$ is given in \eqref{eq:coreSMH}. 
For the latter, in turn, we will prove the conditions of the Aldous-Rebolledo criterion.

First, the tightness of $\{F(\nun_t)\}_{n}$ for fixed $t$ follows directly from 
\eqref{eq:compactContainmentBalls}.\par
Second, by Lemma \ref{le:generatorLMVPfiniteCase}, $N_t^{(n)} \coloneqq F(\nun_t) - \int_0^t \Gall_n  F (\nun_s)ds$ defines a martingale
so that the finite variation process of $\left(F(\nun_t)\right)_{t\geq0}$ is given by
$$
V^{(n)}_t = \int_0^t \Gall_n F (\nun_s)ds.
$$
On the other hand, the compensator of the process $\stp{(N_t^{(n)})^2}$ is given by
$$
\QV{N^{(n)}}_t=\int_0^t \CDC_n F(\nun_s) ds
$$
where $\CDC_n$ is the carré du champ operator associated to $\Gall_n$. The latter is given by
$$
\CDC_n F = \Gall_n F^2 - 2 F \Gall_n F
$$
(see e.g. section 1.2.2 in \cite{JoffeMetivier1986}). The remaining 
conditions of the Aldous-Rebolledo criterion on $\stp{V^{(n)}_t}$ and $\stp{\QV{N^{(n)}}_t}$ (which are given by \eqref{eq:condii.1Tight} and \eqref{eq:condii.2Tight} below) will follow
once we prove
that for
any $F\in \polys(\sM(\typeS))$, $\delta>0$, and any stopping time $0\leq\tau_n\leq T$,

\begin{equation}\label{eq:nuTightnessCondition}
\limsup_{n\to\infty} \sup_{0\leq \theta \leq \delta} \E\[\int_{\tau_n}^{\tau_n+\theta} 
\left(\abs{F(\nun_s)}\vee 1\right) \abs{\Gall_n F(\nun_s)} ds \] \leq  \delta C 
\end{equation}
for some $C>0$ depending only on $F$.
Indeed, 
by Markov's inequality we obtain, on the one hand,
\begin{align*}
    &\limsup_{n\to\infty}\sup_{0\leq \theta\leq \delta}\P\left( \abs{V^{(n)}_{\tau_n+\theta} - V^{(n)}_{\tau_n}} >\eps\right)\\
    &\leq \limsup_{n\to\infty} \sup_{0\leq \theta \leq \delta} \E\[\int_{\tau_n}^{\tau_n+\theta} 
\abs{\Gall_n F(\nun_s)} ds \] \leq   \frac{\delta C}{\eps}.
\end{align*}
So that, taking $\delta=\eps^2/C$, 
\eqref{eq:nuTightnessCondition} implies 
\begin{equation}\label{eq:condii.1Tight}
 \sup_{n\geq n_0}\sup_{0\leq \theta \leq \delta} \P\[\abs{V^{(n)}_{\tau_n +\theta} - V^{(n)}_{\tau_n}}>\eps\]\leq\eps.
 \end{equation}
On the other hand, 
for any $\delta>0$,
\begin{align*}
&\E\left [ \abs{\QV{N^{(n)}}_{\tau_n+\delta} -  \QV{N^{(n)}}_{\tau_n}} \right ]\\
&\leq \E\left [\int_{\tau_n}^{\tau_n+\delta} \abs{\Gall_n F^2(\nun_s)} ds\right] +
2\E\left[
\int_{\tau_n}^{\tau_n+\delta} \abs{ F(\nun_s) \Gall_n F(\nun_s)} ds\right ],
\end{align*}
so that now \eqref{eq:nuTightnessCondition} applied to each term in the r.h.s. above,
and choosing $\delta$ adequately,
 yields 
 \begin{equation}\label{eq:condii.2Tight}
 \sup_{n\geq n_0}\sup_{0\leq \theta \leq \delta} \P\[\abs{\QV{N^{(n)}}_{\tau_n +\theta} - \QV{N^{(n)}}_{\tau_n}}>\eps\]\leq\eps
 \end{equation}
 again through Markov's inequality.\par
It remains to prove \eqref{eq:nuTightnessCondition}, but this is a simple consequence of the fact that $F\in D_\Gall$ is bounded, together
with the uniform bound 
$$
\Gall_n F \leq C \norm{\Lambda_n}_{TV} \leq C\norm{\Lambda}_{TV},
$$
which follos from Lemma \ref{le:boundOnGLambda}.
%

{\bf Step 2: Limiting
martingale problem (existence of solutions).}
Assume that, along a subsequence $\left\{n_i\right\}_{i\geq1}$ (that we will simply denote by $n$ to ease notation) we
have 
$$
\stp{\nun_t}\disttoas{n\to\infty}\stp{\nu_t}
$$
in $D([0,\infty), \sM\left(\typeS\right))$ for some process $\stp{\nu_t}$. 
Lemma \ref{le:boundOnGLambda} provides the necessary conditions for Lemma IV.5.1 in \cite{EthierKurtz86} to be applied and conclude that
$\stp{\nu_t}$ is a solution to the martingale problem for $(\Gall,D_{\Gall}).$ By Corollary \ref{cor:D_Gall'<D_Gall} it is also a solution to the martingale problem for $(\Gall,D_{\Gall}').$ 

\par 

 {\bf Step 3: Identification of the limit and Feller property.}
    The identification of the limit along any weakly convergent subsequence
    is a consequence of the previous Step 2 together with Theorem \ref{th:duality} where uniqueness of solutions to the martingale problem for $\Gall$ is proved (assuming existence of solutions to the martingale problem for $(\Gdall,D_\Gdall)$ which is independently proved in section \ref{sec:dualConstruction}). 
    \par
    For the characterization of each of the coordinates in the limit,
    observe that
computing $\Gall F$ for functions of the form $F(\nu)=h(\log(\norm{\nu}))$ with $h\in\cKdf{2}(\R)$, and 
$F(\nu)=H_\pi^{(\phi)}(\nu)=\mint{\phi}{\left(\frac{\nu}{\norm{\nu}}\right)^{\otimes p}}$ with $ \phi\in\cbf(\typeS^p)$ respectively
--using e.g. \eqref{eq:duality} and \eqref{eq:dualGenerator}--; we obtain
the generators of the Lévy process with characteristic triplet $(\driftF+\drift,\sigma,\LevyMeasure)$ and the $(\Lambda+\sigma^2\delta_0)$-Fleming-Viot process 
 respectively. The result then follows from the uniqueness of solutions to their corresponding martingale problems.\par

We now prove that 
$\stp{\nu_t}$ is Feller. First note that 
by Proposition 3.3 in \cite{Foucart2012} the Fleming-Viot process 
$\stp{\rho_t}=\stp{\nu_t/\norm{\nu_t}}$ is Feller, as well as the Lévy process $\stp{\xi_t}=\stp{\log\norm{\nu_t}}$.
Then, to prove the continuity in probability  of $\stp{\nu_t}$ at $t=0$, 
note that $\rho_t\probtoas{t\to0}\rho_0$ and
$\xi_t\probtoas{t\to0}\xi_0$ imply $(\rho_t,\xi_t)\probtoas{t\to0}(\rho_0,\xi_0)$,
so that $\nu_t=e^{\xi_t}\rho_t\probtoas{t\to0}\nu_0$. We now prove the weak continuity of $\stp{\nu_t}$ as a function of the initial state $\nu_0$.
Consider a sequence of processes $\left\{\stp{\nun_t}\right\}_{n}$ started at $\nun_0$ where $\nun_0\toas{n\to\infty}\nu_0$. 
By the same argument as in Step 1 (Tightness) above, with $\Lambda_n$ therein set to $\Lambda_n=\Lambda$, 
the sequence $\left\{\stp{\nun_t}\right\}_{n}$ is tight. Furthermore, now using Step 2 (limiting martingale problem), the limit along any weakly-convergent subsequence of $\left\{\stp{\nun_t}\right\}_{n}$ is also identified in this case 
to be the unique solution to the martingale problem for $\Gall$ started at $\nu_0$. Then $\stp{\nun_t}\disttoas{n\to\infty}\stp{\nu_t}$ follows.\\

Finally, the fact that $\stp{\nu_t}$ is SMH follows from verifying item {\it\ref{itTh:smhGenerator}} in Proposition \ref{prop:smhssGenerators}. Given 
\eqref{eq:DWgenIsSS}, we only do this for the term $\Gjump{\Lambda}$. This is simply given by 
observing that, for $b>0$, we have $( \SSS_{b}F)'(\nu;a)=bF'(b\nu;a)$ so that 
\begin{align*}
\Gjump{\Lambda}\SSS_{b}F(\nu)&=\int_{\typeS} \frac{\nu(da)}{\norm{\nu}} \int_{{(0,1)}}\frac{\Lambda(d\zeta)}{\zeta^2} \bigg\{ F\left(b\nu + b\norm{\nu}\frac{\zeta}{1-\zeta} \delta_a\right)\\
      &\quad\quad- F(b\nu) - b\norm{\nu}(\abs{\log(1-\zeta)}\Ind{\zeta<1/2}) F'(b\nu; a)\bigg\}.
\end{align*}
Then
$$
\SSS_{b^{-1}}\Gjump{\Lambda}\SSS_{b}F(\nu)=\Gjump{\Lambda}\SSS_{b}F(b^{-1}\nu)=\Gjump{\Lambda}F(\nu).
$$
\end{proof}

We now gather results and prove Theorem \ref{th:MVPsmhExist}.
\begin{proof}[Proof of Theorem \ref{th:MVPsmhExist}]
The existence of the Feller process $\stp{\nu_t}$ with generator of the form $\Gall$ in $D_\Gall$, and the characterization and MAP property of $(\rho_t,\xi_t)$,
are given in Proposition \ref{th:weakLimitMartingale}. The fact that $D_\Gall'\subset D_\Gall$ is proved in Corollary \ref{cor:D_Gall'<D_Gall}. Uniqueness
of solutions to the martingale problem for $(\Gall, D_\Gall)$ is given by Theorem \ref{th:duality}.
 \end{proof}

\begin{remark}
We conjecture that Theorems \ref{th:MVPsmhExist} and \ref{th:mvpLamperti} could be generalized in the following two directions using the same arguments in the proof.
\begin{itemize}
\item[i)] A general positive SS Markov process, see \cite{BertoinKortchemski2016}, can be attained for the total mass  process $\stp{\norm{\mu_t}}$
by adding negative jumps of the form $x\to x\zeta$, $0<\zeta<1,$ with intensity $\Theta^{-}(d\zeta)$,
the {pushforward of a Lévy measure $\Pi$ on $(-\infty,0)$ (i.e. a measure satisfying $\int_{-\infty}^0( \zeta^2\wedge 1) \LevyMeasure(d\zeta)<\infty$) under the
transformation $\zeta \to e^\zeta$. 
A simple way to do this is to 
allow for the underlying SMH process $\stp{\nu_t}$ to have jumps of the form $\nu\to \nu \zeta$ with intensity $\Theta^{-}(d\zeta)$. 
In this sense,
 a term of the form
$$
\Gjumpneg{\Theta^{-}}F(\nu) = \int_{(0,1)} \Theta^{-}(d\zeta)\left\{
F\left(\nu \zeta\right)-F(\nu)
\right\}
$$
is added to its generator.}
Note that adding negative jumps in this way does not change the dynamics of the corresponding frequency process,
so that the arguments leading to the construction of the SMH process $\stp{\nu_t}$ can be easily generalized in this direction.
\item[ii)]  It is possible to extend our results to the
$\Xi$-Fleming-Viot case with $\int_{[0,1]^\N} \sum_{k=1}^\infty \zeta_k^2\Xi(d\zeta)<\infty$. 
 In this case 
the generator of $\stp{\nu_t}$ is given by updating the jumping part $\Gjump{}$ as
\begin{align*}
\Gjump{\Xi}F(\nu)
&=\int_{\left(\typeS\right)^\N} \left(\frac{\nu}{\norm{\nu}}\right)^{\otimes\N}(da)
  \int_{{(0,1)}^\N}\Xi(d\zeta)  \bigg\{ F\left(\nu + \norm{\nu} \sum_{i=1}^\infty \frac{\zeta_i}{1-\norm{\zeta}_{\lp{1}}} \delta_{a_i}\right) - F(\nu)\\
&\qquad\qquad\qquad  - \norm{\nu}\abs{\log(1-\norm{\zeta}_{\lp{1}})}\Ind{\norm{\zeta}_{\lp{1}}<1/2} F'(\nu; a)\bigg\}.
\end{align*}
Heuristically, jumps are of the form
$$
\nu \to \nu + \sum_{k=1}^\infty \frac{\zeta_i}{1-\norm{\zeta}_{\lp{1}}} \delta_{a_k}
$$
where the positions of the new atoms $(a_k)_{k\geq 0}$ are 
i.i.d with distribution $\frac{\nu_{t-}}{\norm{\nu_{t-}}}$ and
the atom sizes $\zeta=(\zeta_1, \zeta_2, \dots)$, satisfying $\norm{\zeta}_{\lp{1}}< 1$ arrive with intensity $dt\otimes\Xi(d\zeta)$.

The process $\stp{\rho_t}$ becomes the $\Xi$-Fleming-Viot process and $\stp{\xi_t}$ is a Lévy process with compensated jumps given by
$$
\int_0^t\int_0^1 \abs{\log(1-\norm{\zeta}_{\lp{1}})} \tilde{\PPP}(ds, d\zeta)
$$
where $\tilde{\PPP}(ds, d\zeta) = \PPP(ds, \zeta) - \Ind{\norm{\zeta}_{\lp{1}}<1/2}ds\otimes\Xi(d\zeta)$ 
and $\PPP$ is a Poisson point process on $\R^+\times (0,1)^\N$ with intensity $ds\otimes\Xi(d\zeta)$.
\end{itemize}\end{remark}
\par

\section{Construction of the process $\stp{\Pi_t,Z_t}$.}\label{sec:dualConstruction}

 For each $n\geq 2$, let  $\tilde\Lambda_n$ be the measure $\Lambda$ restricted to $(\frac{1}{n+1}, \frac{1}{n})$ and
 consider independent Poisson point processes $\mathcal{\tilde P}_n$ on $(0,\infty) \times (0,1)$ with respective intensity measures $dt \otimes \zeta^{-2}\tilde\Lambda_n(d\zeta)$.
 Observe that all $\mathcal{\tilde P}_n$ have finitely many atoms  in any bounded time interval.
     For each $n$, we can construct a coalescent process $\stp{\Pi_t^{(n)}}$ in $\PS{\infty}$ with Kingman's dynamics of intensity $\sigma^2$, and 
     with multiple merging of lineages via independent coin tossing, with respective probabilities and times given by the atoms of the joint Poisson point processes $\cup_{i=1}^n \mathcal{\tilde P}_i$. 
     Of course each coalescent process  $\stp{\Pi_t^{(n)}}$ has characteristic measure $\sigma^2\delta_0+\Lambda_n$ where $\Lambda_n=\sum_{i=2}^n\tilde\Lambda_i$ is the measure $\Lambda$ restricted to $(\frac1n,\frac12)$.
     Note that, as we did before in the construction of the process $\stp{\nu_t}$, here we are leaving out the finitely-many ``big'' jumps of $\zeta^{-2}\Lambda(d\zeta)$ in $[1/2,1)$. As before, these can be added back to the process $\stp{\Pi_t,Z_t}$
     using similar arguments as for $\stp{\nu_t}$ (see e.g. the proof of Lemma \ref{le:generatorLMVPfiniteCase}).
     On the other hand, we also assume that the Kingman's components of all $\stp{\Pi_t^{(n)}}$  are driven by the same Poisson point measure on $\Rp\times\PS{\infty}$, following the Poissonian construction of general coalescent processes introduced in \cite{Bertoin2006}.
     
\begin{lemma}\label{lem:cauchy}
The sequence $\{\Pi^{(n)}\}_n$ is Cauchy  in $L^3$ in the following sense.
	For any $T\geq 0$, 
	$$
	\lim_{n,m\to\infty}\E\left[\sup_{0\leq t\leq T} \dist_{\PS{\infty}}(\Pi_t^{(n)}, \Pi_t^{(m)})^3\right] =0,
	$$
		where $\dist_{\PS{\infty}}$ is the metric in $\PS{\infty}$ defined by $\dist_{\PS{\infty}}(\pi,\hat\pi)=1/\max\left\{p\colon \restr{\pi}{[p]}=\restr{\hat\pi}{[p]}\right\}$.
\end{lemma}
\begin{proof}
Set $n> m$ and 
observe that $\stp{\Pi_t^{(m)}}$ and $\stp{\Pi_t^{(n)}}$ have common jumps determined by the same paintbox partitions driven by the atoms of $\cup_{i=1}^m \mathcal{\tilde P}_i$, as well as a common Kingman component.
 On the other hand, $\stp{\Pi_t^{(n)}}$ has extra jumps determined by the atoms of $\cup_{i=m+1}^n \mathcal{\tilde P}_{i}$.
Hence the distance between both processes is related to these extra jumps.
More precisely, denote these extra atoms by $(t_i,\zeta_i), i\leq K$  
  and consider a family of independent negative binomial random variables $M_1,\dots, M_K$ such that $M_i$ has parameters 2 and $\zeta_i$.
   Its probability mass function is given by $\P(M_i= p)=(p-1) \zeta_i^2(1-\zeta_i)^{p-2}$, for $p\geq2$.
The variable $M=\min M_i$ provides a lower bound for the value $p$ such that 
$\restr{\Pi_{t}^{(n)}}{[p]}=\restr{\Pi_{t}^{(m)}}{[p]}$ for all $t\leq T$, since it gives the label of the second block, in increasing order, participating into a coalescence event produced by the extra atoms.
 Hence,
    $1/M$ is an a.s. upper bound for 
   $\sup_{0\leq t\leq T} \dist_{\PS{\infty}}(\Pi_t^{(n)}, \Pi_t^{(m)})$.
   
      Now, the collection $\{(t_i,M_i), i	\leq K\}$ is a Poisson point process on $[0,T]\times \{2,3,\cdots\}$ with intensity 
    $dt\otimes \int_{1/n}^{1/m} (p-1) \zeta^2(1-\zeta)^{p-2} \zeta^{-2}dp\Lambda(d\zeta)$ where $dp$ refers to the counting measure on $\{2,3,\cdots\}$.
    From the latter we obtain
    \begin{align*}
    \E[M^{-3}] &\leq \E\left[\sum_{i=1}^K M_i^{-3}\right] = 
    \int_0^T dt\int_{\frac{1}{n}}^{\frac{1}{m}}  \frac{\Lambda(d\zeta)}{\zeta^2} 
    \sum_{p\geq2} p^{-3} (p-1) \zeta^2(1-\zeta)^{p-2} \\
    &= \int_0^T dt\int_{\frac{1}{n}}^{\frac{1}{m}}   \Lambda(d\zeta)
    \sum_{p\geq2}  \frac{p-1}{p}p^{-2} (1-\zeta)^{p-2} \leq T\sum_{p\geq2}  p^{-2}{\Lambda}\left(\left[\frac{1}{n},\frac{1}{m}\right]\right).
    \end{align*}
and this quantity converges to 0 as $n,m\to\infty$.
\end{proof}

\begin{proof}[Proof of Theorem \ref{th:FellerDual}]
  Recall the independent Poisson point processes $(\mathcal{\tilde P}_n)_n$ on $(0,\infty)\times(0,1)$ with respective intensity $dt \otimes \zeta^{-2}\tilde\Lambda_n(d\zeta)$. 
  Fix $n\geq1$, set $\mathcal{ P}_n=\cup_{k=1}^n\mathcal{\tilde P}_k$ and denote its atoms by $(t_i, \zeta_i)$.
  Consider the  coupled process $(\Pi_t^{(n)}, \xi_t^{(n)})$  constructed from the same Poisson measure $\mathcal{ P}_n$.
     The coalescent coordinate $\Pi_t^{(n)}$ evolves through and independent Kingman component as well as by the merging of lineages at times $t_i$ according to independent coin tossing of probability $\zeta_i$.
     The  Lévy coordinate $\xi_t^{(n)}$ is standardly defined as in \cite{Bertoin1996}:
    $$
    \xi_t^{(n)} =(\drift-\sigma)t+\sigma B_t-\sum_{t_i \leq t} \log(1-\zeta_i)+t\int_0^\infty\frac{\log(1-\zeta)}{\zeta^2}1_{\{\zeta<1/2\}}\Lambda_n(d\zeta),
    $$
    where $\stp{B_t}$ is an independent Browninan motion.
    It is clear from its dynamics that the process $(\Pi_t^{(n)}, Z_t^{(n)})$  with $Z_t^{(n)}=e^{\xi_t^{(n)}}$,  is solution to the martingale problem associated to the generator \eqref{eq:dualGenerator} (replacing $\Lambda$ by $\Lambda_n$).
   
The limit as $n\to\infty$ is obtained thanks to Cauchy sequence arguments.
From Lemma \ref{lem:cauchy}, the sequence of processes $\left\{\stp{\Pi^{(n)}_t}\right\}_n$ is Cauchy in $L^3$.
From the proof of Theorem 1.1 in \cite{Bertoin2006}, the sequence of processes $\left\{\stp{\xi^{(n)}_t}\right\}_n$ is Cauchy in $L^2$.
Hence by defining the distance on $\PS{\infty}\times\mathbb R$ by
$$\dist((\pi,z), (\pi',z'))=\dist_{\PS{\infty}}(\pi,\pi')+|z-z'|
$$
we can show that the sequence  $\left\{\stp{\Pi^{(n)}_t, \xi^{(n)}_t}\right\}_n$ is Cauchy in $L^3$. Indeed,
	$$
	\E\left[\sup_{0\leq t\leq T} \dist((\Pi_t^{(n)},\xi_t^{(n)}), (\Pi_t^{(m)}, \xi_t^{(m)}))^3\right] \leq 4\E\left[\sup_{0\leq t\leq T} \dist_{\PS{\infty}}(\Pi_t^{(n)}, \Pi_t^{(m)})^3\right] +4	\E\left[\sup_{0\leq t\leq T} |\xi_t^{(n)}-\xi_t^{(m)}|^3\right] 
	$$
which converges to 0 as $n,m\to \infty$. As a consequence, the processes $\stp{\Pi^{(n)}_t, Z^{(n)}_t}$
{converge in distribution as $n\to\infty$. By Corollary \ref{cor:convergenceDualGenerators} and Lemma IV.5.1 in \cite{EthierKurtz86}  the limit $\stp{\Pi_t, Z_t}$ is a solution to the martingale problem associated to $(\Gdall, D_\Gdall)$ in \eqref{eq:dualGeneratorDomain}-\eqref{eq:dualGenerator}.}\par
Uniqueness and the Feller property are obtained with similar arguments as for $\stp{\nu_t}$ in the proof of Theorem \ref{th:MVPsmhExist}.
\end{proof}
	
{\bf Funding:}\\
Arno Siri-Jégousse was supported by DGAPA-PAPIIT-UNAM grant IN-105726 and the PASPA program of DGAPA-UNAM.\\
Alejandro H. Wences was supported by the ANR LabEx CIMI (grant ANR-11-LABX-0040) within the French State Programme “Investissements d’Avenir.”

\bibliographystyle{plain} 
\bibliography{Math}

@book{Bertoin1996,
	author    = {Bertoin, Jean},
	title     = {L{\'e}vy Processes},
	series    = {Cambridge Tracts in Mathematics},
	volume    = {121},
	year      = {1996},
	publisher = {Cambridge University Press},
	address   = {Cambridge},
	isbn      = {0-521-56243-0}
}

@book{KyprianouPardo2022book, 
place={Cambridge}, 
series={Institute of Mathematical Statistics Monographs}, 
title={{Stable Lévy processes via Lamperti-type representations}}, 
publisher={Cambridge University Press}, 
author={Kyprianou, Andreas E. and Pardo, Juan C.}, 
year={2022}, collection={Institute of Mathematical Statistics Monographs}}

@article{BertoinKortchemski2016,
 ISSN = {10505164},
 URL = {http://www.jstor.org/stable/24810065},
 author = {Bertoin, Jean  and Kortchemski, Igor},
 journal = {The Annals of Applied Probability},
 number = {4},
 pages = {2556--2595},
 publisher = {Institute of Mathematical Statistics},
 title = {{Self-similar scaling limits of Markov chains on the positive integers}},
 volume = {26},
 year = {2016}
}

@article{CasanovaNunezPerez2023,
      author={González Casanova, Adrián  and Nuñez, Imanol  and Pérez, José L.},
title = {{Alpha-stable branching and beta-frequency processes, beyond the IID assumption}},
volume = {29},
journal = {Electronic Communications in Probability},
publisher = {Institute of Mathematical Statistics and Bernoulli Society},
pages = {1 -- 12},
year = {2024},
doi = {10.1214/23-ECP570},
URL = {https://doi.org/10.1214/23-ECP570}
}

@article{JohnstonLambert2023,
author = {Johnston, Samuel G. G.  and Lambert, Amaury },
title = {{The coalescent structure of uniform and Poisson samples from multitype branching processes}},
volume = {33},
journal = {The Annals of Applied Probability},
number = {6A},
publisher = {Institute of Mathematical Statistics},
pages = {4820 -- 4857},
year = {2023},
doi = {10.1214/23-AAP1934},
URL = {https://doi.org/10.1214/23-AAP1934}
}

@article{BirknerBlathEtal2009,
  title={{A modified lookdown construction for the Xi-Fleming-Viot process with mutation and populations with recurrent bottlenecks}},
  author={Birkner, Matthias  and Blath, Jochen  and M\"{o}hle, Martin  and Steinrücken, Matthias  and Tams, Johanna},
    journal = {ALEA Lat. Am. J. Probab. Math. Stat.},
  fjournal = {ALEA Latin-American Journal of Probability and Mathematical Statistics},
  year={2009},
  url={https://api.semanticscholar.org/CorpusID:2169486}
}

@article{HassMiermont2011,
author = {Haas, B{\'e}n{\'e}dicte  and Miermont, Gr{\'e}gory },
title = {{Self-similar scaling limits of non-increasing Markov chains}},
volume = {17},
journal = {Bernoulli},
number = {4},
publisher = {Bernoulli Society for Mathematical Statistics and Probability},
pages = {1217 -- 1247},
year = {2011},
doi = {10.3150/10-BEJ312},
URL = {https://doi.org/10.3150/10-BEJ312}
}

@article{GCOSJ,
author = {Gonz{\'a}lez Casanova, Adri{\'a}n  and Offenstadt, Ariel  and Siri-J{\'e}gousse, Arno} ,
title = {Moment duality and propagation of exchangeability},
journal = {Preprint on Arxiv},
year = {2026}
}

@article{CasanovaMiroSchertzerJegousse2024,
author = {Gonz{\'a}lez Casanova, Adri{\'a}n  and Mir{\'o} Pina, Ver{\'o}nica  and Schertzer, Emmanuel  and Siri-J{\'e}gousse, Arno} ,
title = {{Asymptotics of the frequency spectrum for general Dirichlet Xi-coalescents}},
volume = {29},
journal = {Electronic Journal of Probability},
publisher = {Institute of Mathematical Statistics and Bernoulli Society},
pages = {1 -- 35},
year = {2024},
doi = {10.1214/23-EJP1064},
URL = {https://doi.org/10.1214/23-EJP1064}
}

@article{ChaumontPantiRivero2013,
author = { Chaumont, Lo{\"i}c and Pant{\'i}, Henry  and Rivero, V{\'i}ctor },
title = {{The Lamperti representation of real-valued self-similar Markov processes}},
volume = {19},
journal = {Bernoulli},
number = {5B},
publisher = {Bernoulli Society for Mathematical Statistics and Probability},
pages = {2494 -- 2523},
year = {2013},
doi = {10.3150/12-BEJ460},
URL = {https://doi.org/10.3150/12-BEJ460}
}

@article{CaballeroCasanovaPerez2023, 
      author={ Caballero, María E. and González Casanova, Adrián  and Pérez, José L.},
title = {{The relative frequency between two continuous-state branching processes with immigration and their genealogy}},
volume = {34},
journal = {The Annals of Applied Probability},
number = {1B},
publisher = {Institute of Mathematical Statistics},
pages = {1271 -- 1318},
year = {2024},
doi = {10.1214/23-AAP1991},
URL = {https://doi.org/10.1214/23-AAP1991}
}

@article{DonnellyKurtz1999,
 ISSN = {00911798},
 URL = {http://www.jstor.org/stable/2652871},
 author = {Peter Donnelly and Thomas G. Kurtz},
 journal = {The Annals of Probability},
 number = {1},
 pages = {166--205},
 publisher = {Institute of Mathematical Statistics},
 title = {Particle representations for measure-Valued population models},
 urldate = {2024-02-16},
 volume = {27},
 year = {1999}
}

@article{JoffeMetivier1986,
 ISSN = {00018678},
 URL = {http://www.jstor.org/stable/1427238},
 author = {Joffe, Anatole  and Metivier, Michel},
 journal = {Advances in Applied Probability},
 number = {1},
 pages = {20--65},
 publisher = {Applied Probability Trust},
 title = {Weak convergence of sequences of semimartingales with applications to multitype branching processes},
 urldate = {2024-01-18},
 volume = {18},
 year = {1986}
}

@article {DawsonHochberg1982,
    AUTHOR = {Dawson, Donald A. and Hochberg, Kenneth J.},
     TITLE = {Wandering random measures in the {F}leming-{V}iot model},
   JOURNAL = {The Annals of Probability},
  FJOURNAL = {The Annals of Probability},
    VOLUME = {10},
      YEAR = {1982},
    NUMBER = {3},
     PAGES = {554--580},
      ISSN = {0091-1798,2168-894X},
   MRCLASS = {92A10 (60G57 60J70)},
  MRNUMBER = {659528},
       URL =
              {http://links.jstor.org/sici?sici=0091-1798(198208)10:3<554:WRMITF>2.0.CO;2-K&origin=MSN}
}

@article{Foucart2012,
 author = { Foucart, Clément},
  journal = {ALEA Lat. Am. J. Probab. Math. Stat.},
  fjournal = {ALEA Latin-American Journal of Probability and Mathematical Statistics},
 number = {9},
 pages = {451--472},
 publisher = {Instituto Nacional de Matem{\'a}tica Pura e Aplicada},
 title = {{Generalized Fleming-Viot processes with immigration via stochastic flows of partitions}},
 volume = {2},
 year = {2012}
}

@article{KyprianouPardo2008,
 ISSN = {00219002},
 URL = {http://www.jstor.org/stable/27596014},
 author = { Kyprianou, Andreas E. and Pardo, Juan C.},
 journal = {Journal of Applied Probability},
 number = {4},
 pages = {1140--1160},
 publisher = {Applied Probability Trust},
 title = {Continuous-state branching processes and self-similarity},
 urldate = {2023-11-27},
 volume = {45},
 year = {2008}
}

@book{ZenghuLi2022,
  title={{Measure-valued branching Markov processes}},
  author={Zenghu Li},
  isbn={978-3-662-66910-5},
  series={Series Title Probability Theory and Stochastic Modelling},
  year={2022},
  publisher={Springer Berlin, Heidelberg},
  doi = {https://doi.org/10.1007/978-3-662-66910-5}
}

@article{Sawyer1970,
 ISSN = {00029947},
 URL = {http://www.jstor.org/stable/1995636},
 author = {Sawyer, Stanley A. },
 journal = {Transactions of the American Mathematical Society},
 number = {1},
 pages = {1--38},
 publisher = {American Mathematical Society},
 title = {A formula for semigroups, with an application to branching diffusion processes},
 urldate = {2023-10-15},
 volume = {152},
 year = {1970}
}

@Inbook{Fitzsimmons1992,
author="Fitzsimmons, Patrick J.",
title="On the martingale problem for measure-valued Markov branching processes",
bookTitle="Seminar on Stochastic Processes, 1991",
year="1992",
publisher={Birkh{\"a}user Boston},
address="Boston, MA",
pages="143--156",
isbn="978-1-4612-0381-0",
url={https://doi.org/10.1007/978-1-4612-0381-0_12}
}

@Inbook{Perkins1992,
author="Perkins, Edwin A.",
title="Conditional Dawson-Watanabe processes and Fleming-Viot processes",
bookTitle="Seminar on Stochastic Processes, 1991",
year="1992",
publisher="Birkh{\"a}user Boston",
address="Boston, MA",
pages="143--156",
isbn="978-1-4612-0381-0",
url="https://doi.org/10.1007/978-1-4612-0381-0_12"
}

@inbook{Dawson1993,
author = {Dawson, Donald A.},
title = {{Measure valued Markov processes}},
booktitle = {Hennequin, PL. (eds) Ecole d'Eté de Probabilités de Saint-Flour XXI - 1991},
year = {1993},
series = {Lecture Notes in Mathematics},
volume = {1541},
publisher = {Springer, Berlin, Heidelberg},
doi = {https://doi.org/10.1007/BFb0084190}
}

@incollection{PerkinsBolthausenVaart2004,
  author    = {Perkins, Edwin A. and Bolthausen, Erwin and van der Vaart, Aad},
  title     = {Dawson-{W}atanabe superprocesses and measure-valued diffusions},
  booktitle = {Lectures on Probability Theory and Statistics: {E}cole d'Et{\'e} de Probabilit{\'e}s de {S}aint-{F}lour {XXIX} - 1999},
  series    = {Lecture Notes in Mathematics},
  volume    = {1781},
  pages     = {125--329},
  year      = {2004},
  publisher = {Springer-Verlag},
  address   = {Berlin, Heidelberg},
  editor    = {Bernard, Pierre},
  isbn      = {978-3-540-43736-9},
  doi       = {10.1007/b83371}
}

@book{GetoorBlumenthal1968,
  title={Markov processes and potential theory},
  author={Blumenthal, Robert M. and Getoor, Ronald K.},
  series={Pure and Applied Mathematics},
  year={1968},
  publisher={Academic Press, New York and London }
}

@book{Cohn2013,
author = {Cohn, Donald L.},
address = {New York, NY},
booktitle = {Measure theory},
isbn = {978-1-4614-6955-1},
publisher = {Birkhäuser},
series = {Birkhäuser Advanced Texts Basler Lehrbücher},
title = {Measure theory},
year = {2013}
}

@article{AliliChaumont2017,
author = {Alili, Larbi and Chaumont, Loïc  and Graczyk, Piotr  and  Zak, Tomasz},
title = {{Inversion, duality and Doob $h$-transforms for self-similar Markov processes}},
volume = {22},
journal = {Electronic Journal of Probability},
publisher = {Institute of Mathematical Statistics and Bernoulli Society},
pages = {1 -- 18},
year = {2017},
doi = {10.1214/17-EJP33},
URL = {https://doi.org/10.1214/17-EJP33}
}

@book{Bertoin2006, 
place={Cambridge}, 
series={Cambridge Studies in Advanced Mathematics}, 
title={Random fragmentation and coagulation processes}, 
DOI={10.1017/CBO9780511617768}, 
publisher={Cambridge University Press}, 
author={Bertoin, Jean}, year={2006}, 
collection={Cambridge Studies in Advanced Mathematics},
}

@article{BertoinLeGall2000,
author = {Bertoin, Jean and Le Gall, Jean-François},
year = {2000},
pages = {249-266},
title = {The {B}olthausen–{S}znitman coalescent and the genealogy of continuous-state branching processes},
volume = {117},
journal = {Probability Theory and Related Fields},
doi = {10.1007/s004400050006}
}

@Article{BertoinLeGall2003,
author={Bertoin, Jean and Le Gall, Jean-François},
title={Stochastic flows associated to coalescent processes},
journal={Probability Theory and Related Fields},
year={2003},
day={01},
volume={126},
number={2},
pages={261-288},
issn={1432-2064},
doi={10.1007/s00440-003-0264-4},
url={https://doi.org/10.1007/s00440-003-0264-4}
}

@book{Billingsley99,
	address = {New York},
	edition = {2nd ed},
	series = {Wiley series in probability and statistics. {Probability} and statistics section},
	title = {Convergence of probability measures},
	isbn = {9780471197454},
	publisher = {Wiley},
	author = {Billingsley, Patrick},
	year = {1999},
}

@book{etheridge2000introduction,
  title={An introduction to superprocesses},
  author={Etheridge, Alison},
  isbn={9780821827062},
  lccn={00044160},
  series={University lecture series},
  url={https://books.google.to/books?id=zUHyBwAAQBAJ},
  year={2000},
  publisher={American Mathematical Society}
}

@Book{EthierKurtz86,
 author = {Ethier, Stewart N. and Kurtz, Thomas G.},
 title = {Markov processes: characterization and convergence},
 publisher = {John Wiley \& Sons},
 year = {1986},
 address = {New York},
 isbn = {978-0-470-31732-7}
}

@article{GnedinIksanovMarynych2014,
author = {Gnedin, Alexander and Iksanov, Alexander and Marynych, Alexander},
year = {2014},
pages = {23-40},
title = { {$\Lambda$}-coalescents: a survey},
volume = {51A},
journal = {Journal of Applied Probability},
doi = {10.1239/jap/1417528464}
}

@article{Kingman1982,
title = {The coalescent},
journal = {Stochastic Processes and their Applications},
volume = {13},
number = {3},
pages = {235-248},
year = {1982},
issn = {0304-4149},
doi = {https://doi.org/10.1016/0304-4149(82)90011-4},
url = {https://www.sciencedirect.com/science/article/pii/0304414982900114},
author = {Kingman, John}
}

@Article{Lamperti1972,
author={Lamperti, John},
title={{Semi-stable Markov processes. I}},
journal={Zeitschrift f{\"u}r Wahrscheinlichkeitstheorie und Verwandte Gebiete},
year={1972},
day={01},
volume={22},
number={3},
pages={205-225},
issn={1432-2064},
doi={10.1007/BF00536091},
url={https://doi.org/10.1007/BF00536091}
}

@article{les7,
author = {Birkner,Matthias  and Blath,Jochen  and Capaldo,Marcella  and Etheridge,Alison  and Möhle,Martin  and Schweinsberg,Jason  and Wakolbinger,Anton },
title = {{Alpha-stable branching and beta-coalescents}},
volume = {10},
journal = {Electronic Journal of Probability},
publisher = {Institute of Mathematical Statistics and Bernoulli Society},
pages = {303 -- 325},
year = {2005},
doi = {10.1214/EJP.v10-241},
URL = {https://doi.org/10.1214/EJP.v10-241}
}

@article{Pitman99,
author = {Pitman,Jim },
title = {{Coalescents with multiple collisions}},
volume = {27},
journal = {The Annals of Probability},
number = {4},
publisher = {Institute of Mathematical Statistics},
pages = {1870 -- 1902},
year = {1999},
doi = {10.1214/aop/1022874819},
URL = {https://doi.org/10.1214/aop/1022874819}
}

@article{Sagitov99, 
title={The general coalescent with asynchronous mergers of ancestral lines},
 volume={36}, 
DOI={10.1239/jap/1032374759}, 
number={4}, 
journal={Journal of Applied Probability}, 
publisher={Cambridge University Press}, 
author={Sagitov, Serik}, 
year={1999}, 
pages={1116–1125}}

@article{Schweinsberg2003,
title = {{Coalescent processes obtained from supercritical Galton–Watson processes}},
journal = {Stochastic Processes and their Applications},
volume = {106},
number = {1},
pages = {107-139},
year = {2003},
issn = {0304-4149},
doi = {https://doi.org/10.1016/S0304-4149(03)00028-0},
url = {https://www.sciencedirect.com/science/article/pii/S0304414903000280},
author = {Schweinsberg, Jason }
}
\end{document}